\documentclass[11pt]{amsart}
\usepackage{amssymb}
\usepackage{amsmath}
\usepackage{graphicx}
\usepackage[english]{babel}
\usepackage{color}
\usepackage{enumerate}
\usepackage[ansinew]{inputenc}
\usepackage{multicol}

\usepackage{pb-diagram}

\swapnumbers

\theoremstyle{plain}
\newtheorem{theorem}[subsection]{Theorem}
\newtheorem{proposition}[subsection]{Proposition}
\newtheorem{lemma}[subsection]{Lemma}

\theoremstyle{definition}
\newtheorem{definition}[subsection]{Definition}
\newtheorem{examples}[subsection]{Examples}
\newtheorem{example}[subsection]{Example}
\newtheorem{remark}[subsection]{Remark}
\newtheorem{notation}[subsection]{Notation}

\def\pun{\cdot}
\def\cir{\circ}

\def\KK{{\mathbb{K}}}

\begin{document}

\author[S. M\'arquez]{Sebasti\'an M\'arquez}
\address{Instituto de Matem\'aticas y F\'isica,
Universidad de Talca, Avda. Lircay s/n, Talca, Chile}
\email{semarquez@utalca.cl}

\title{COMPATIBLE ASSOCIATIVE BIALGEBRAS}

\subjclass[2000]{ .}
\thanks{The author\rq s work was partially supported by the Project  FONDECYT Regular 1130939 and by MathAmSud 17Math-05 LIETS.}


\begin{abstract}
We introduce a non-symmetric operad $\mathcal{N}$, whose dimension in degree $n$ is given by the Catalan number $c_{n-1}$. It arises naturally in the study of coalgebra structures defined on compatible associative algebras.
We prove that any free compatible associative algebra admits a compatible infinitesimal bialgebra structure, whose subspace of primitive elements is a $\mathcal{N}$-algebra. The data $({\rm As},{\rm As}^2, \mathcal{N})$ is a good triple of
operads, in J.-L. Loday's  sense. Our construction induces another triple of operads $({\rm As},{\rm As}_2,{\rm As})$, where ${\rm As}_2$ is the operad of matching dialgebras.

Motivated by A. Goncharov's Hopf algebra of paths $P(S)$, we introduce the notion of bi-matching dialgebras and show that the Hopf algebra $P(S)$ is a bi-matching dialgebras.
\end{abstract}

\title{COMPATIBLE ASSOCIATIVE BIALGEBRAS}

\maketitle
\section*{Introduction} \label{s:int}

In the present work, we study the existence of coalgebra structures on compatible associative algebras.
A {\it compatible associative algebra} over a field $\KK$ is a vector space $A$ equipped  with two associative products,  $\cdot$  and $\circ$, satisfying that the sum
$$(1)\ x \ast y:=x \cdot y + x \circ y, $$
is an associative product.\\

In \cite{Strohmayer-Operads of compatible structures}, H. Strohmayer developed the general notion of compatible algebraic structures, and denotes the operad of compatible associative algebras by $\rm{As}^2$. He computed the Koszul dual of ${\rm As}^2$, denoted by $^2{\rm As}$, which is a set theoretical operad, by arising an operadic partition poset (see \cite{ParttionPosets}). The same author showed, using B. Valette\rq s results, that $^2{\rm As}$ is a Koszul operad, and therefore $\rm{As}^2$ is Koszul operad, too.

Applying H. Strohmayer\rq s work, V. Dotsenko  obtained in \cite{Dotsenko-Compatible associative product}  the dimensions of the operad $\rm{As}^2$ and computed the characters of $\rm{As}^2(n)$, both as an $S_n$-module and as an $S_n\times SL_2$-module. The dimension of $\rm{As}^2(n)$ is $n!$ times the Catalan number $c_n$.

The Catalan numbers are sequence of natural numbers, named after the Belgian mathematician Eug\`ene Charles Catalan (1814-1894). There exist many counting problems in combinatorics whose solution is given by them, a large description of the combinatorial objects described by Catalan numbers may be found in \cite{Stanley}. In particular, the $n^{th}$ Catalan number describes the number of plane rooted trees.\medskip

We look for coalgebra structures which are compatible with condition $(1)$, and therefore may be defined on any free compatible associative algebra. In \cite{LodayRonco-On the structure of cofree Hopf algebras}, J.-L. Loday and M. Ronco introduced the notion of unital infinitesimal bialgebra, as an associative unital algebra $(C, *, u)$ equipped with a coassociative coproduct $\Delta: C\longrightarrow C\otimes C$ satisfying that :
$$\Delta (x*y)= \sum x_{(1)}\otimes (x_{(2)}*y) + (x*y_{(1)})\otimes y_{(2)}\ -\ x\otimes y,$$
for $x,y\in C$, where $\Delta (x)=\sum x_{(1)}\otimes x_{(2)}$ and $\Delta (y)=\sum y_{(1)}\otimes y_{(2)}$.

A {\it compatible infinitesimal bialgebra} is a compatible algebra $(A, \cdot, \circ)$ equipped with a coassociative coproduct $\Delta$, satisfying the unital infinitesimal relation with both associative products.

We give an explicit construction of free objects in the category of $\rm{As}^2$-algebras, easier to work with than the one described in \cite{Dotsenko-Compatible associative product}.  Using it, we describe a canonical coproduct $\Delta$ on any free $\rm{As}^2$-algebra, which satisfies the unital infinitesimal condition with both products.

Following J.-L. Loday (see \cite{Loday-GeneralizedBialgebras}) and R. Holtkamp (see \cite{Holtkamp}), we know that there exists an algebraic operad which describes the subspace of primitive elements of a compatible associative bialgebra, we call this new structure a ${\mathcal N}$-algebra. The operad ${\mathcal N}$ is non-symmetric, and the dimension of the $\KK$-vector space ${\mathcal N}_n$ is the Catalan number $c_{n-1}$. We prove that:\begin{enumerate}
\item  there exists an operad homomorphism from  ${\mathcal N}$ to the operad of compatible associative algebras,
\item the subspace of primitive elements of any compatible associative bialgebra has a natural structure of ${\mathcal N}$ algebra,
\item the free compatible associative algebra over a vector space $V$ is isomorphic, as a coalgebra, to the cofree conilpotent coalgebra spanned by the free ${\mathcal N}$ algebra over $V$.\end{enumerate}

As a consequence of the previous results, we prove that the category of conilpotent compatible associative algebras is equivalent to the category of ${\mathcal N}$ algebras, which gives a good triple of operads $({\rm As},{\rm As}^2, \mathcal{N})$
as defined by J.-L. Loday in \cite{Loday-GeneralizedBialgebras}.
\medskip

A {\it matching dialgebra}, previously studied by C. Bai, L. Guo and Y.Zhang in  \cite{Zhang-The category and operad of matching dialgebras}, is a compatible associative algebra $(A, \cdot, \circ)$ satisfying that:
$$(x \cdot y) \circ z=x\cdot (y \circ z) \text{ and } (x\circ y) \cdot z=x\circ (y \cdot z), $$
for all the elements $x,y,z \in A$.

Motivated by {\it the path algebra} introduced by A. Goncharov  in \cite{Goncharov-Galois symmetries of fundamental groupoids and noncommutative geometry}, we define the notion of {\it bi-matching dialgebras} ,and show that the path algebra is a example of a bi-matching dialgebra, which is obtained from {\it semi-homomorphism} of algebras.
Finally, studying the subcategory of compatible associative bialgebras satisfying that their underlying compatible associative algebra is a matching dialgebra, we obtain another triple of operads $({\rm As},{\rm As}_2,{\rm As})$.
\medskip

The manuscript is organized as follows: Section 1 we construct the free compatible associative algebra over a vector space $V$.

In Section 2 we develop the notion of compatible infinitesimal bialgebra, while in Section 3 we introduce the notion of ${\mathcal N}$-algebra and prove the structure theorem for compatible infinitesimal bialgebras.

Matching dialgebras are defined in Section 4, where we describe A. Goncharov\rq s path Hopf algebra as the main example of this type of structure.\\
\subsection*{Acknowledgment}
I would like to express my thanks to Prof. M. Ronco for motivating me to work on this problem and for her constant contributions to it, and to Prof. A. Labra for many useful comments and for encouraging me to continue my research work. My special thanks to the University of Talca for the support provided during this period.

\bigskip
\section*{Notations}

All vector spaces and algebras considered in the manuscript are over a field $\KK$. Given a set $X$, we denote by $\KK[X]$ the vector space spanned by $X$. For any vector space $V$, we denote by $V^{\otimes n}$ the tensor product of $V\otimes \dots \otimes V$, $n$ times, over $\KK$. In order to simplify notation, we shall denote an element of $V^{\otimes n}$ by $x_1\cdots x_n$.\medskip

If $n$ is a positive integer, we denote by $[n]$ the set $\{1,\ldots, n\}$. The symmetric group of permutations of $[n]$ is denoted by $S_n$. Given a permutation $\sigma \in S_n$, we write $\sigma=(\sigma(1),\ldots,\sigma(n))$, identifying $\sigma$ with its image.

\section{The free compatible associative algebra}\label{freeCompAssAlg}
\medskip

In \cite{OperadLieComp}, V. Dotsenko and A. Khoroshkin computed the dimensions of the components for the operad of the {\it compatible Lie algebras} and for the {\it bi-Hamiltonian operad}.

A {\it compatible Lie algebra} is a $\KK$-vector space $A$, equipped  with two Lie brackets $[,]$ and $\{,\}$ , satisfying that their sum is also a Lie bracket.

When the compatible Lie algebra $A$ is equipped with a commutative and associative product $\cdot$ such that the brackets are both derivations for the product $\cdot$, we say that $A$ is a  {\it bi-Hamiltonian algebra}. In particular, a bi-Hamiltonian algebra is an analogue of a double structure of Poisson algebra, which appears naturally in certain examples of integrable systems.

As in the classical case, there exists a functor from the category of $\rm{As}^2$-algebras over $\KK$ to the category of compatible Lie algebras over the same field. If $(A, \cdot, \circ)$ is a compatible associative  algebra,  then the Lie brackets given by
\[\begin{array}{rll} [x,y]&=x \cdot y -y\cdot x , \\
\{x,y\}&=x \circ y -y\circ x, \\
\end{array}\]
define a structure of compatible Lie algebra on the underlying vector space of $A$.

In this section, we construct the free compatible associative algebra on a vector space $V$ by means of planar rooted trees. We recall the definition of compatible associative algebra, already described in Introduction.

\begin{definition}\label{def:AssComp} A \emph{compatible associative algebra} is a vector space $A$ together with two associative products $\pun :A\otimes A\to A$ and $\cir:A\otimes A\to A$, satisfying that their sum $*:=\pun + \cir $ is an associative product, too.
\end{definition}

Note that the if the product $*=\pun + \cir $ is  associative, then all linear combinations $\lambda\pun\  +\ \mu\ \cir $ are associative products, for any coefficients $\lambda$ and $\mu $ in $\KK$.

\begin{remark}\label{compatibilidad}  The condition $\pun + \cir $ is an associative product, is equivalent to:
$$x\cir(y\pun z)+x\pun(y\cir z)=(x\cir y)\pun z+(x\pun y)\cir z,$$
for all elements $x, y, z \in A$.\end{remark}

\medskip

V. Dotsenko in \cite{Dotsenko-Compatible associative product} constructed associative products on vector spaces spanned by trees using R. Grossman and R.G. Larson\rq s constructions (see \cite{Grossman-larson_Hopf-algebraic structure of families of trees}), but the associative products are compatible only in certain cases and quite difficult to deal with.\medskip

We give a different construction of the free associative compatible algebra, applying Dotsenko\rq s results on its dimensions, by means of planar rooted trees.
\subsection{The vector space ${\rm As}^2(V)$}
Let $V$ be a vector space with basis $X=\{a_i\}_{i \in I}$. Denote by $T{}^{X}_{n}$ the set of planar rooted trees with $(n+1)$ vertices, whose vertices different from the root are colored by the elements of $X$. For instance:

\begin{figure}[h!]
  \centering
  \includegraphics[width=12cm]{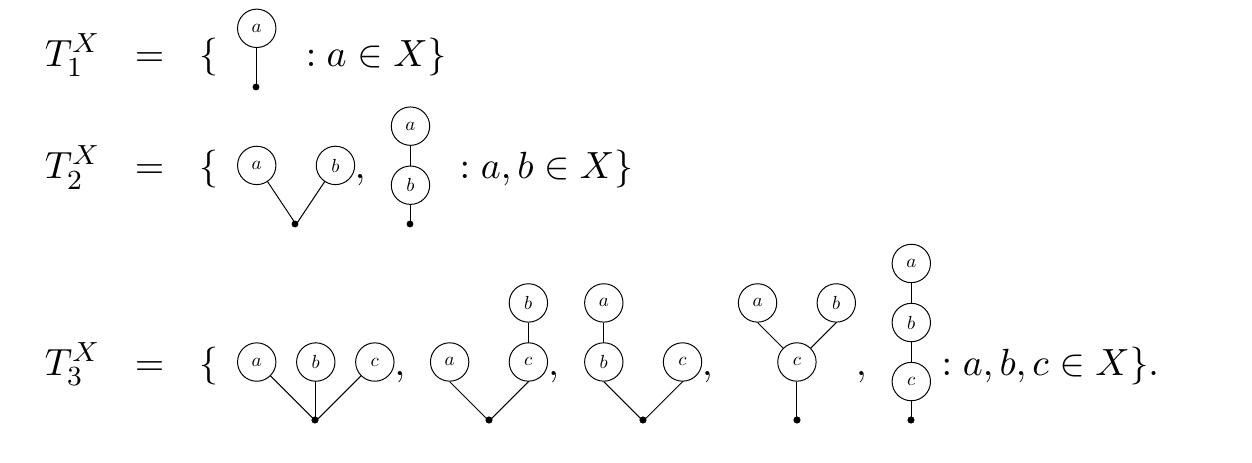}  
\end{figure}

Consider the vector space ${\rm{As}}^2(V)=\KK[\bigcup_{n\geq 1}T{}^{X}_{n}]=\bigoplus_{n\geq
1}^{} \KK[T{}^{X}_{n}]$, whose basis is the set $\bigcup_{n\geq 1}T{}^{X}_{n}$ of all planar rooted colored trees.

\begin{remark}

For any tree $t$ in $T{}^{X}_{n}$, we say that $t$ has degree $n$ and we write $\vert t\vert =n$.
We consider the tree $t$ oriented from bottom to top.

Given a vertex $v \in t$, we say that a vertex $v' \in t$ is a child of $v$ if $v'$ is directly connected to the vertex $v$.
\end{remark}

\begin{notation}
Given a tree $t$, the set of vertices of $t$ is denoted by ${\rm Vert}(t)$ and the root of $t$ by ${\rm root}(t)$. The subset ${\rm Vert}(t)\setminus \{{\rm root}(t)\}$ of ${\rm Vert}(t)$ is denoted by ${\rm Vert}^{\ast}(t)$.
\end{notation}

We define two associative products in ${\rm As}^2(V)$.

\begin{definition}\label{ProductoPunto}
Let $t, w$ be trees in ${\rm{As}}^2(V)$. Define $t \pun w$ as the tree obtained by identifying the roots of $t$ and $w$. Extending this binary operation by linearity, we get an associative product $$\cdot :{\rm{As}}^2(V) \otimes {\rm{As}}^2(V) \rightarrow {\rm{As}}^2(V).$$
\end{definition}

\begin{remark}\label{arbolesIrreEn$As^2$}
Note that any tree $t$ in ${\rm{As}}^2(V)$ may
be written in a unique way as $t=t^{1}\pun \ldots \pun t^{r}$, where $r \geq 1$ and the root of each $t^{i}$ has only one child, for each $i \in \{1,\ldots, r\}$. Clearly, we have that $|t|=\displaystyle{\sum_{i=1}^{r}|t^i|} $.

When the root of a tree $t \in T^X_n$ has a unique child, we say that $t$ is {\it irreducible}. We identify the elements of the basis $X$ with the trees of degree one (which are irreducible).

Denoting by ${\rm Irr}$ the vector space spanned by the set of all irreducible trees in ${\rm{As}}^2(V)$, we have that $({\rm{As}}^2(V),\cdot)$ is free over ${\rm Irr}$ as an associative algebra. The set of irreducible trees of degree $n$ is denoted by ${\rm Irr}_n$.
\end{remark}

\subsection{Second product}\label{ProductoCirculo}
Let  $t$ and $w$ be trees in ${\rm{As}}^2(V)$, with $t=t^1 \pun \ldots \pun t^r$ as described in Remark \ref{arbolesIrreEn$As^2$}. A second product $t\cir w$ is defined proceeding by induction on the degree $n$ of $w$.

If $n=1$, then $w=a$, for some  $a \in X$. In this case, the element $t \circ w $ is the tree obtained by replacing the root of $t$ by the vertex, colored with $a$, and adding a new root.\\

\begin{figure}[h!]
  \centering
  \includegraphics[width=6cm]{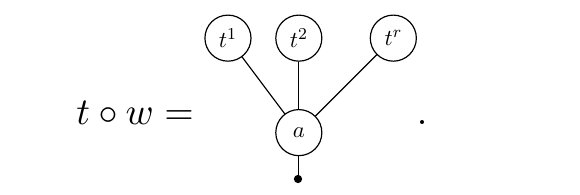}
\end{figure}

Assume now that $n > 1$, and that the product $t \circ w$ has been defined for any $|w| < n$.

Let $w=w^{1} \pun \ldots \pun w^{m}$ be the unique decomposition of $w$ as a product of irreducible trees.

If $m=1$, then $w=w^{1}=u \circ a$, where $u$ is a tree such that $|u|=n-1$ and $a$ is an element of the basis $X$. Applying a recursive argument, we may suppose that $t \circ u$ is already defined. The product $t \circ w$ is the element
$$t \cir w :=t \cir (u \cir a ) = ( t \cir u ) \cir a .$$

For $m > 1$, the element $t \circ w$ is defined by the following formula:

\begin{equation*}
\begin{split}
t \circ w &= \sum _{i=1}^{m}((t \cdot w^1 \cdot \ldots \cdot w^{i-1})\circ w^{i})\cdot \ldots \cdot  w^{m}\\
& \quad -\sum_{i=2}^{m} t\cdot ((w^{1} \pun \ldots \pun w^{i-1})\circ w^{i})\cdot w^{i+1}\cdot \ldots \cdot w^{m}\\
\end{split}
\end{equation*}

Note that the recursive hypothesis states that each term of the previous formula is well defined.

\begin{example}
Let $a_1, \ldots , a_n$ be elements of the basis $X$, with $n\geq 2$. Consider the tree $w$
given by\\\medskip

\begin{figure}[h!]
  \centering
  \includegraphics[width=9cm]{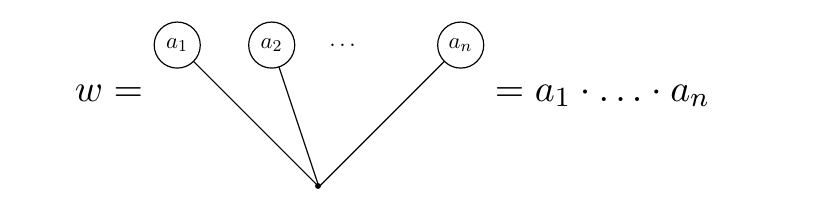}
\end{figure}

\medskip
If $n=2$, then we have that
$$t \cir w =t \cir (a_1 \pun a_2)= (t
\cir a_1) \pun a_2 - t \pun (a_1 \cir a_2) + (t \pun a_1
)\cir a_2, $$
therefore we get

\begin{figure}[h!]
  \centering
  \includegraphics[width=8cm]{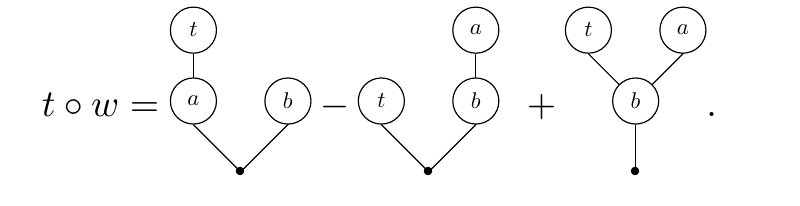}
\end{figure}

In general, for $w=a_1\cdot \ldots \cdot a_n$, we get the formula:
\medskip

\begin{figure}[h!]
  \centering
  \includegraphics[width=10cm]{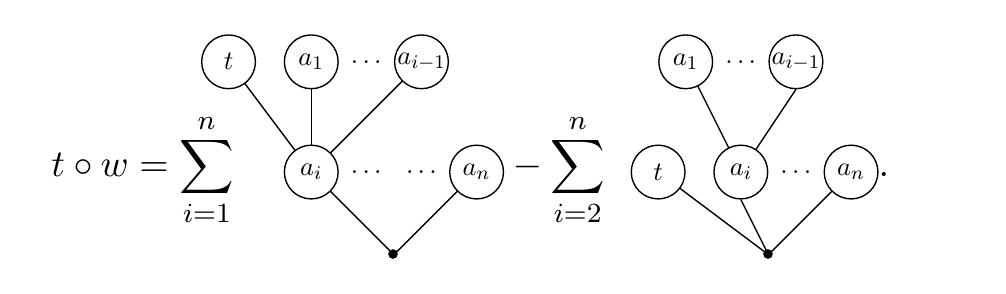}
\end{figure}

\end{example}
\medskip

\begin{proposition} The vector space ${\rm As}^2(V)$, equipped with the products $\pun $ and $\cir$ defined in \ref{ProductoPunto} and \ref{ProductoCirculo}, is the free associative compatible algebra on $V$.
\end{proposition}

\begin{proof}
Note that the definition of the product $\circ$ implies that the products $\cdot$ and $\circ$ satisfy the compatibility condition.

Let us prove that the product $\circ$ is associative. Let $t_1, t_2, t_3$ be trees in ${\rm As}^2(V)$. To see that $t_1\circ (t_2 \circ t_3)= (t_1\circ t_2) \circ t_3$, we proceed by induction on the degree of the third term $t_3$.

If the degree of $t_3$ is one, the assertion is follows easily from Definition \ref{ProductoCirculo}.

Suppose that $n=|t_3| > 1$. If $t_3$ is an irreducible tree, then $t_3=w \circ a$, where $w$ is a tree with $|w|=n-1$, and $a$ is an element of the basis $X$. Applying a recursive argument, we get the following identities:
\[\begin{array}{rll} t_1 \circ (t_2 \circ t_3)&=t_1 \circ (t_2 \circ (w \circ a))& \\
&=t_1 \circ ((t_2 \circ w) \circ a)\\
&=(t_1 \circ (t_2 \circ w)) \circ a\\
&=((t_1 \circ t_2) \circ w) \circ a\\
&=(t_1 \circ t_2) \circ (w \circ a)\\
&=(t_1 \circ t_2) \circ t_3,\\
\end{array}\]
which imply the result.

Suppose now that $t_3=w \pun z$, where $w$ and $z$ are trees of  degree smaller than $|t_3|$.

By the compatibility condition and a recursive argument, we get that:
\[\begin{array}{rll}
t_1 \circ (t_2 \circ t_3)&=t_1\circ (t_2 \circ (w\pun z))& \\
&=t_1\circ ((t_2 \circ w)\pun z -t_2\pun(w \circ z)+ (t_2 \pun w)\circ z) \\
&=(t_1\circ(t_2\circ w))\pun z-t_1\pun ((t_2\circ w)\circ z)+(t_1\pun (t_2 \circ w))\circ z \\
&\quad -(t_1\circ t_2 )\pun (w \circ z)+ t_1 \pun (t_2 \circ (w \cir z))-(t_1 \pun t_2 )\circ (w \circ z )\\
&\quad +t_1\circ ((t_2 \pun w)\circ z)\\
&=((t_1 \circ t_2) \circ w)\pun z-(t_1 \circ t_2 )\pun (w \circ z )\\
&\quad +(t_1 \circ (t_2 \pun w )-(t_1 \pun t_2 ) \circ w + t_1 \pun (t_2 \circ w))\circ z.\\
\end{array}\]

As $(t_1\circ t_2)\pun w = t_1 \circ (t_2 \pun w )-(t_1 \pun t_2 ) \circ w + t_1 \pun (t_2 \circ w)$, we conclude that

\[\begin{array}{rll}
t_1 \circ (t_2 \circ t_3)&=((t_1 \circ t_2) \circ w)\pun z-(t_1 \circ t_2 )\pun (w \circ z )+ ((t_1\circ t_2)\pun w)\circ z& \\
&=(t_1 \circ t_2)\circ (w \pun z )\\
&=(t_1 \circ t_2)\circ t_3.\\
\end{array}\]

To end the proof, we need to see that ${\rm As}^2(V)$ is free as associative compatible algebra. Let $A$ be an associative compatible algebra and let $f :V \rightarrow A$ be
a linear map. The homomorphism $\widetilde{f} :{\rm As}^2(V) \rightarrow A$ is defined in a recursive way.
\medskip

Let $t$ be a tree in ${\rm As}^2(V)$. If $|t|=1$ then $t=a$ with $a \in X$ and therefore $\widetilde{f}(t)=f(a)$.
\medskip

Suppose that $|t|>1$. If $t=t'\circ a$, for some $a \in X$, is irreducible,  we define $$\widetilde{f}(t)=\widetilde{f}(t')\circ f(a),$$ which is well defined by a recursive argument.

If $t=t^1 \pun \ldots \pun t^r$ for some $r >1$, then we can assume that $\widetilde{f}(t^i)$ is defined, for $1\leq i \leq r$, and set $\widetilde{f}(t)=\widetilde{f}(t^1) \pun \ldots \pun \widetilde{f}(t^r)$.\medskip

To see that $\widetilde{f}$ is unique, consider $g:{\rm As}^2(V) \rightarrow A$, a homomorphism of compatible associative algebras such that $g(a)=f(a)$, for $a \in V$. Let $t$ be a tree in ${\rm As}^2(V)$. If $|t|=1$, then $t=a$, with $a \in X$. So, by definition of $\widetilde{f}$, $g(t)=\widetilde{f}(t)$. Suppose that $|t|>1$. We have that $t=t'\circ a$, for some $a \in X$, or $t=t^1 \pun \ldots \pun t^r$, for some $r >1$. Applying a recursive argument, we have that $g(t)=\widetilde{f}(t)$. This show that $\widetilde{f}$ is unique, which ends the proof.

\end{proof}

\section{Compatible infinitesimal bialgebras}\label{SeccionComInfBial}
In this section, we introduce {\it compatible infinitesimal bialgebras}, which uses the notion of unital infinitesimal bialgebra introduced by J.-L Loday and M. Ronco in \cite{LodayRonco-On the structure of cofree Hopf algebras}. To work in the more general context, we do not assume the existence of unity. So, an infinitesimal bialgebra is an associative algebra $(H, \cdot)$ equipped with a coassociative coproduct $\Delta :H\longrightarrow H\otimes H$ satisfying
$$\Delta(x \cdot y)=  x_{(1)}\otimes (x_{(2)} \cdot y) + (x \cdot y_{(1)})\otimes y_{(2)} + x\otimes y,$$
for $x,y\in H$, with $\Delta (x)= x_{(1)}\otimes x_{(2)}$  and $\Delta(y)=y_{(1)}\otimes y_{(2)}$ for $x, y \in H$.\medskip

An element $x \in H$ is called {\it primitive} when $\Delta(x)=0$.

\begin{definition} \label{InfComp}{\rm A {\it  compatible infinitesimal bialgebra} over $\KK$ is an associative compatible algebra $(H, \pun ,\cir )$ equipped with a coassociative coproduct $\Delta :H\longrightarrow H\otimes H$ such that $(H,\pun, \Delta )$ and $(H,\cir ,\Delta )$ are both unital infinitesimal bialgebras.}\end{definition}

\begin{lemma}\label{BuenaDefUIAA} The notion of compatible infinitesimal bialgebra is well-defined.\end{lemma}

\begin{proof} A direct computation shows that:
\[\begin{array}{rl} (1) \quad \Delta((x\cdot y)\circ z) &= \quad x_{(1)}\otimes (x_{(2)}\cdot y)\circ z+x\cdot y_{(1)}\otimes y_{(2)}\circ z+x\otimes (y\circ z) \\
& +(x\cdot y)\circ z_{(1)}\otimes z_{(2)}+(x\cdot y)\otimes z,\\
(2) \quad \Delta((x\circ y)\cdot z) &= \quad x_{(1)}\otimes (x_{(2)}\circ y)\cdot z+x\circ y_{(1)}\otimes y_{(2)}\cdot z+x\otimes (y\cdot z) \\
& +(x\circ y)\cdot z_{(1)}\otimes z_{(2)}+(x\circ y)\otimes z,\\
(3) \quad \Delta(x\cdot (y\circ z)) &= \quad x_{(1)}\otimes x_{(2)}\cdot( y\circ z)+x\cdot y_{(1)}\otimes y_{(2)}\circ z+x\otimes (y\circ z) \\
& +x\cdot (y\circ z_{(1)})\otimes z_{(2)}+(x\cdot y)\otimes z,\\
(4) \quad \Delta(x\circ( y\cdot z)) &= \quad x_{(1)}\otimes x_{(2)}\circ( y\cdot z)+x\circ y_{(1)}\otimes y_{(2)}\cdot z+x\otimes (y\cdot z) \\
& +x\circ (y\cdot z_{(1)})\otimes z_{(2)}+(x\circ y)\otimes z.\\
\end{array}\]
Using the compatibility condition between the products $\cdot$ and $\circ$, we get that:
\begin{enumerate}
  \item $x_{(1)}\otimes (x_{(2)}\cdot y)\circ z+x_{(1)}\otimes (x_{(2)}\circ y)\cdot z= x_{(1)}\otimes x_{(2)}\cdot( y\circ z)+x_{(1)}\otimes x_{(2)}\circ( y\cdot z),$

  \item $(x\cdot y)\circ z_{(1)}\otimes z_{(2)}+(x\circ y)\cdot z_{(1)}\otimes z_{(2)}=x\cdot (y\circ z_{(1)})\otimes z_{(2)}+x\circ (y\cdot z_{(1)})\otimes z_{(2)},$
\end{enumerate}
which implies that $$\Delta((x\cdot y)\circ z+(x\circ y)\cdot z)=\Delta(x\cdot (y\circ z)+x\circ( y\cdot z)).$$
\end{proof}
\medskip

\begin{proposition}\label{$As^2(V)$ ComoCompInfBialgebra} Let $V$ be a vector space, the free associative compatible algebra $\rm{As}^2(V)$ has a natural structure of compatible infinitesimal bialgebra.\end{proposition}

\begin{proof}
The coproduct $\Delta : {\rm As}^2(V)\rightarrow {\rm As}^2(V) \otimes {\rm As}^2(V)$ is defined by induction on  the degree of a tree $t$ in ${\rm As}^2(V)$.

For $t=a \in X$, its image is $\Delta(t)=0$. When $|t|>1$, we consider two cases:
\begin{enumerate}
\item for $t=t' \circ a$, with $a \in X$, we define  $$\Delta(t)=t'_{(1)}\otimes t'_{(1)}\circ a + t' \otimes a.$$
\item for $t=t' \pun t''$ with $|t'|<|t| \text{ and }|t''|< |t|$, we have that $$\Delta(t)=t'_{(1)}\otimes t'_{(2)}\pun t'' + t' \pun t''_{(1)}\otimes t''_{(2)}+t' \otimes t''.$$
\end{enumerate}

Lemma \ref{BuenaDefUIAA} and the inductive hypothesis, state that $\Delta$ is well defined. Note that if $\Delta(t)=t_{(1)} \otimes t_{(2)}$ then  $|t_{(1)}|<|t| \text{ and }|t_{(2)}| < |t|$.

To see that $\Delta$ is coassociative, we proceed by induction on degree of $t$. Let $t$ be a tree. For $|t|=1$ the result is immediate.\\
For  $|t|>1$, we consider two case:

First, if $t$ is an irreducible tree, then  $t=t'\circ a$, with $a \in X$. So, we have that:
\[\begin{array}{rll}
(\Delta \otimes {\rm Id})\Delta (t)&=(\Delta \otimes {\rm Id})(t'_{(1)}\otimes t'_{(2)} \circ a+t'\otimes a)& \\
&=\Delta(t'_{(1)})\otimes t'_{(2)} \circ a+\Delta(t')\otimes a \\
&=t'_{(1)(1)}\otimes t'_{(1)(2)}\otimes t'_{(2)}\circ a+t'_{(1)}\otimes t'_{(2)}\otimes a. \\
\end{array}\]

Applying the recursive hypothesis to $t'$, we  write
$$(\Delta \otimes {\rm Id})\Delta (t)=t'_{(1)}\otimes t'_{(2)}\otimes t'_{(3)}\circ a+t'_{(1)}\otimes t'_{(2)}\otimes a.$$

On the other hand, using a similar argument to computer $({\rm Id} \otimes \Delta)\Delta(t)$, we have that
\[\begin{array}{rll}
({\rm Id} \otimes \Delta)\Delta (t)&=({\rm Id}\otimes \Delta)(t'_{(1)}\otimes t'_{(2)} \circ a+t'\otimes a& \\
&=t'_{(1)}\otimes \Delta (t'_{(2)} \circ a)+t'\otimes \Delta (a) \\
&=t'_{(1)}\otimes t'_{(2)(1)}\otimes t'_{(2)(2)}\circ a+t'_{(1)}\otimes t'_{(2)}\otimes a \\
&=t'_{(1)}\otimes t'_{(2)}\otimes t'_{(3)}\circ a+t'_{(1)}\otimes t'_{(2)}\otimes a ,\\
\end{array}\]
which gives the expected result.

Second, if $t$ is reducible tree, then $t=t'\pun t''$, with $|t'|<|t| \text{ and }|t''|< |t|$. Applying the recursive hypothesis to $t'$ and $t''$, we have that
\[\begin{array}{rl}
(\Delta \otimes {\rm Id})\Delta (t)&=\quad (\Delta\otimes {\rm Id} )(t'_{(1)}\otimes t'_{(2)} \pun t''+t'\pun t''_{(1)}\otimes t''_{(2)}+t'\otimes t'') \\
&=\quad t'_{(1)(1)}\otimes t'_{(1)(2)}\otimes t'_{(2)} \pun t''+t'_{(1)}\otimes t'_{(2)}\pun t''_{(1)}\otimes t''_{(2)} \\
& + t'\pun t''_{(1)(1)}\otimes t''_{(1)(2)}\otimes t''_{(2)}+t'\otimes t''_{(1)}\otimes t''_{(2)}+ t'_{(1)}\otimes t'_{(2)}\otimes t''\\
&=\quad t'_{(1)}\otimes t'_{(2)}\otimes t'_{(3)} \pun t''+t'_{(1)}\otimes t'_{(2)}\pun t''_{(1)}\otimes t''_{(2)}\\
& + t'\pun t''_{(1)}\otimes t''_{(2)}\otimes t''_{(3)}+t'\otimes t''_{(1)}\otimes t''_{(2)}+ t'_{(1)}\otimes t'_{(2)}\otimes t''\\
\end{array}\]
and,
\[\begin{array}{rl}
({\rm Id} \otimes \Delta)\Delta (t)&=\quad ({\rm Id}\otimes \Delta)(t'_{(1)}\otimes t'_{(2)} \pun t''+t'\pun t''_{(1)}\otimes t''_{(2)}+t'\otimes t'') \\
&=\quad t'_{(1)}\otimes t'_{(2)(1)}\otimes t'_{(2)(2)} \pun t''+t'_{(1)}\otimes t'_{(2)}\pun t''_{(1)}\otimes t''_{(2)} \\
& + t'_{(1)}\otimes  t'_{(2)}\otimes t''+t'\cdot t''_{(1)}\otimes t''_{(2)(1)}\otimes t''_{(2)(2)}+ t'\otimes t''_{(2)}\otimes t''_{(2)}\\
&=\quad t'_{(1)}\otimes t'_{(2)}\otimes t'_{(3)} \pun t''+t'_{(1)}\otimes t'_{(2)}\pun t''_{(1)}\otimes t''_{(2)} \\
& + t'_{(1)}\otimes  t'_{(2)}\otimes t''+t'\cdot t''_{(1)}\otimes t''_{(2)}\otimes t''_{(3)}+ t'\otimes t''_{(2)}\otimes t''_{(2)}.\\
\end{array}\]

So, we conclude that
$$(\Delta \otimes {\rm Id})\Delta (t)=({\rm Id} \otimes \Delta)\Delta (t),$$ which ends the proof.

\end{proof}

\subsection{Formula for the coproduct $\Delta$}
We want to give an explicit formula for the coproduct $\Delta$, for which we previously describe an order on the vertices of a tree.

Given a tree $t$, we consider the set ${\rm Vert}(t)$ ordered by the {\it level order}, that is, the vertices of $t$ are ordered by reading the vertices of $ t $ from left to right and from top to bottom. For instance, if $t$ is the tree

\begin{figure}[h!]
  \centering
  \includegraphics[width=6cm]{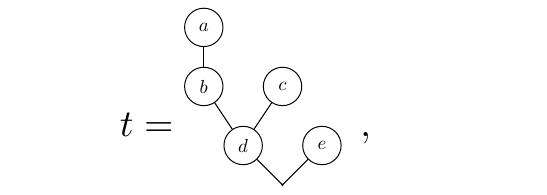}
\end{figure}

then ${\rm Vert}(t)$ is ordered by $a<b<c<d<e<{\rm root}(t)$.

\begin{notation}
Given a tree $t$ and a vertex $v \in {\rm Vert}^{\ast}(t)$, we denote by $e_{v}$ the edge
of $t$ with final vertex $v$.
\end{notation}

\begin{definition}\label{ContraccionDunArbol}
Let $A=\{v_1,\ldots,v_l\}$ be a ordered subset of ${\rm Vert}^{\ast}(t)$ with $v_1< \cdots <v_l$ respect to the level order.

We define $t_A$ as the tree of the degree $l$ such that ${\rm Vert}^{\ast}(t_A)=A$, which is obtained from $t$ by deletions of the vertices that are not in $A$ and by successive contractions of the edges of which these are the final vertex.
\end{definition}

\newpage

\begin{example}
When $t$ is the tree

\begin{figure}[h!]
  \centering
  \includegraphics[width=8cm]{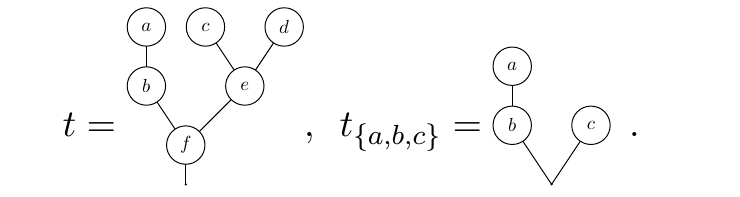}
\end{figure}

\end{example}

\begin{proposition}\label{FormulaCoproducto}
Let $t$ be a tree of degree $n$, with ${\rm Vert}^{\ast}(t)$  given by the ordered vertices $a_1<\cdots <a_n$. The coproduct $\Delta(t)$ is given by the formula
$$\Delta(t)=\sum_{i=1}^{n-1}t_{\{a_1, \ldots,a_i\}}\otimes t_{\{a_{i+1}, \ldots,a_{n}\}}.$$
\end{proposition}
\begin{proof}

Let $t$ be a tree of degree $n$ and $a_1<\cdots <a_n$ its vertices, different of the root, ordered by the level order.

We prove the assertion by induction on $n$. For $n=1$, $t=a \in V$. So, $\Delta(a)=0$ and the assertion is true.

For $n>1$, we consider two cases. First, if $t$ is an irreducible, then $t=t'\circ a_n$, where $t'$ is a tree of degree $(n-1)$, with vertices $a_1,\ldots,a_{n-1}$. In this case, note that
$t'=t_{\{a_1,\ldots,a_{n-1}\}}$ and the vertices $a_1,\ldots,a_{n-1}$ preserve the order that they originally had in $t$. Moreover, if $1\leq k \leq l \leq n-1$, then $t'_{\{a_k,a_{k+1},\ldots,a_l\}}=t_{\{a_k,a_{k+1},\ldots,a_l\}}$.\medskip

By definition of coproduct $\Delta$ and by a recursive argument, we have that
$$\begin{array}{rcl}
\Delta(t) & = & \Delta(t'\circ a_n) \\
          & = & \Delta(t')\circ a_n+t'\circ \Delta(a_n)+t'\otimes a_n\\
          & = & \displaystyle \sum_{i=1}^{n-2}t'_{\{a_1,\ldots,a_i\}}\otimes t'_{\{a_{i+1},\ldots,a_{n-1}\}}\circ a_n+t'\otimes a_n  \\
          & = & \displaystyle \sum_{i=1}^{n-2}t_{\{a_1,\ldots,a_i\}}\otimes t_{\{a_{i+1},\ldots,a_{n-1}\}}\circ a_n+t_{\{a_1,\ldots,a_{n-1}\}}\otimes a_n\\
          & = & \displaystyle \sum_{i=1}^{n-1}t_{\{a_1, \ldots,a_i\}}\otimes t_{\{a_{i+1}, \ldots,a_{n}\}},\\
  \end{array}$$
because $t_{\{a_{i+1},\ldots,a_{n-1}\}}\circ a_n= t_{\{a_{i+1},\ldots,a_{n}\}}$, for $i=1,\ldots,n-2$.\medskip

If $t$ is a reducible tree, then we may write $t=t'\cdot t''$, where $t'$ and $t''$ are trees of degree smaller than $n$. If $t'$ is of degree $l$, then ${\rm Ver}(t')=\{a_1,\ldots,a_l\}$ and ${\rm Ver}(t'')=\{a_{l+1},\ldots,a_{l+m}\}$, where $n=l+m$. Note that the vertices of $t'$ and $t''$ preserve the order that they had in $t$. Moreover, $t'=t_{\{a_1,\ldots,a_l\}}$ and $t''=t_{\{a_{l+1},\ldots,a_{l+m}\}}$.\medskip

By definition of the coproduct $\Delta$ and by a recursive argument, we obtain that:
$$\begin{array}{rcl}
\Delta(t) & = & \Delta(t'\cdot t'') \\
          & = & \Delta(t')\cdot t''+t'\cdot \Delta(t'')+t'\otimes t''\\
          & = & \displaystyle \sum_{i=1}^{l-1}t'_{\{a_1,\ldots,a_i\}}\otimes t'_{\{a_{i+1},\ldots,a_{l}\}}\cdot t''\\
          & &+\displaystyle \sum_{j=1}^{m-1}t'\cdot t''_{\{a_{l+1},\ldots,a_{l+j}\}}\otimes t''_{\{a_{l+j+1},\ldots,a_{l+m}\}}+t'\otimes t'' \\
          & = & \displaystyle \sum_{i=1}^{l-1}t_{\{a_1,\ldots,a_i\}}\otimes t'_{\{a_{i+1},\ldots,a_{l}\}}\cdot t''\\
          & &+\displaystyle \sum_{j=1}^{m-1}t'\cdot t''_{\{a_{l+1},\ldots,a_{l+j}\}}\otimes t_{\{a_{l+j+1},\ldots,a_{l+m}\}}+t'\otimes t''.
  \end{array}$$
As $t'_{\{a_{i+1},\ldots,a_{l}\}}\cdot t''=t_{\{a_{i+1},\ldots,a_n\}}$, for $i=1,\ldots,l-1$, and\\
$t'\cdot t''_{\{a_{l+1},\ldots,a_{l+j}\}}= t_{\{a_1,\ldots,a_{l+j}\}}$, for $j=1,\ldots,m-1$, we get
$$\Delta(t)=\sum_{i=1}^{n-1}t_{\{a_1, \ldots,a_i\}}\otimes t_{\{a_{i+1}, \ldots,a_{n}\}},$$
which ends the proof.
\end{proof}
\begin{example}

When $t$ is the tree

\begin{figure}[h!]
  \centering
  \includegraphics[width=8.5cm]{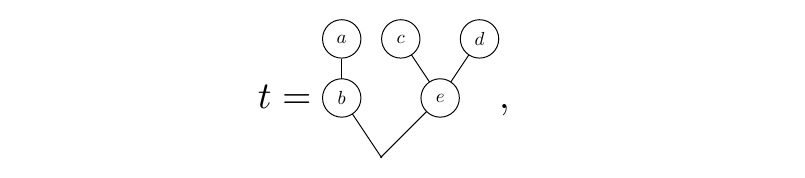}
\end{figure}

the coproduct $\Delta(t)$ is given by

\begin{figure}[h!]
  \centering
  \includegraphics[width=11cm]{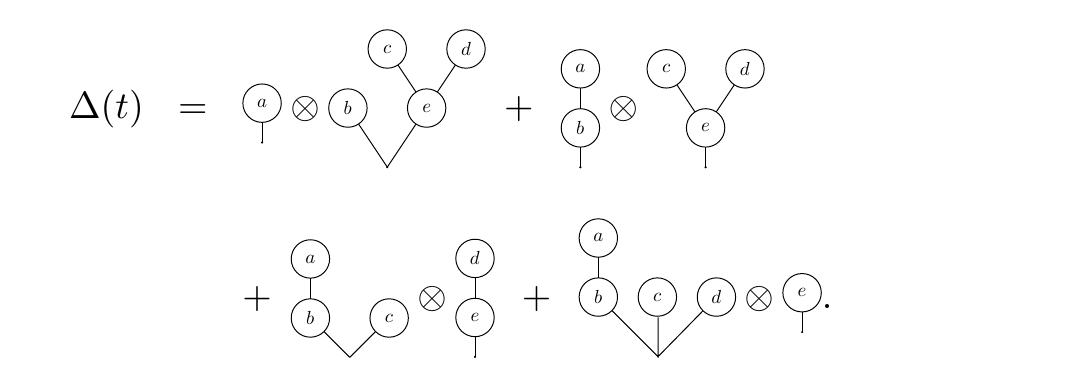}
\end{figure}

\end{example}

\begin{remark}\label{$As^2(v)$EslibreComoAlgAss}
Consider $H_V={\rm As}^2(V)$ with the coproduct $\Delta$ as in Proposition \ref{$As^2(V)$ ComoCompInfBialgebra}. The triples $(H_V,\cdot, \Delta)$ and $(H_V,\circ, \Delta)$ are graded infinitesimal bialgebras, with the natural graduation of $H_V$. In particular, the vector space $V$ is the component of degree one. By definition, $\Delta(v)=0$, for all $v \in V$. So, $(H, \cdot, \Delta)$ and $(H, \circ , \Delta)$ are conilpotent infinitesimal bialgebras.

Applying the result obtained by J.-L. Loday and M. Ronco  in \cite{LodayRonco-On the structure of cofree Hopf algebras}, Theorem 2.6 , we obtain that $H$ is isomorphism to $\overline{T}({\rm Prim } H)$.

In particular, the associative algebras $({\rm As}^2(V), \cdot)$ and $({\rm As}^2(V), \circ)$ are free as associative algebras. This result is an alternative proof to that obtained by  Dotsenko in \cite{Dotsenko-Compatible associative product}, using operad theory .
\end{remark}

\section{Structure theorem for compatible infinitesimal bialgebras}

Our aim is to prove that any conilpotent compatible infinitesimal bialgebra can be reconstructed from the subspace of its primitive elements.

We introduce the notion of $\mathcal{N}$-algebra, which  describes the algebraic structure of the subspace of primitive elements of any compatible infinitesimal bialgebra.

\begin{definition}\label{def:$N$-algebra.}  A {\it $\mathcal{N}$-algebra} is a vector space $V$ equipped with $n$-ary operations $N_{n}:V^{\otimes n}\longrightarrow V$, for $n\geq 2$, which satisfy the following conditions:
\[\begin{array}{l}
(1)\thinspace N_n(x_1,\ldots,N_2(x_n,x_{n+1}))=\displaystyle \sum_{i=1}^{n-1}N_{i+1}(x_1,\ldots,N_{n-i+1}(x_{i},\ldots,x_n),x_{n+1}), \\
(2)\thinspace \text{ for } n \geq 3\\
\quad N_2(x_1,N_n(x_2,\ldots,x_{n+1}))=N_n(N_2(x_1,x_2),x_3,\ldots,x_{n+1})\\
\hspace{1.4cm} -\displaystyle{ \sum_{i=3}^{n} N_i(x_1,N_{n+2-i}(x_2,\ldots ,x_{n+3-i}),x_{n+3-i+1}, \ldots ,x_{n+1})},\\
(3)\thinspace \text{ for } r, n \geq 3\\
\quad N_n(x,y_1\ldots,y_{n-2},N_r(z,t_1,\ldots,t_{r-2},w)) =\\
\hspace{1.4cm} N_r(N_n(x,y_1,\ldots,y_{n-2},z),t_{1},\ldots,t_{r-2},w)\\
\hspace{1.4cm} + \displaystyle{ \sum_{i=1}^{n-2} N_{r+i}(x,y_1,\ldots,y_{i-1},N_{n-i}(y_{i},\ldots ,y_{n-2},z),t_{1},\ldots,t_{r-2},w)}\\
\hspace{1.4cm} - \displaystyle{ \sum_{i=1}^{r-2} N_{n+r-i-1}(x,y_1,\ldots,y_{n-2},N_{i+1}(z,t_1,\ldots ,t_{i}),t_{i+1},\ldots,t_{r-2},w)},\\
\end{array}\]

\end{definition}

For instance, the relations in low degrees give:
\begin{enumerate}
\item $N_2$ is an associative product.
\item $N_3(x,y,N_2(z,t))=N_2(N_3(x,y,z),t)+N_3(x,N_2(y,z),t)$.
\item $N_2(x,N_3(y,z,t))=N_3(N_2(x,y),z,t)-N_3(x,N_2(y,z),t)$.
\item $N_3(x,y,N_3(z,t,w))=N_3(N_3(x,y,z),t,w)+N_4(x,N_2(y,z),t,w)-N_4(x,y,N_2(z,t),w)$.
\end{enumerate}
\begin{remark}
Let $\mathcal{N}$ be the algebraic operad of the $\mathcal{N}$-algebras. It is clear that the operad $\mathcal{N}$ is regular.
So, the $S_n$-module $\mathcal{N}(n)$ is of the form $\mathcal{N}(n)=\mathcal{N}_n \otimes \KK[S_n]$ for some vector space $\mathcal{N}_n$, where $\KK[S_n]$ is the regular representation of $S_n$.
\end{remark}

\begin{proposition}\label{Dimension $N$}
The dimension of the vector space $\mathcal{N}_n$ is equal to the Catalan number $c_{n-1}$.
\end{proposition}
\begin{proof}
Denote by $|\mathcal{N}_n|$ the dimension of $\mathcal{N}_n$, as a $\KK$-vector space. We know that $|\mathcal{N}_n|$ is the dimension of the subspace of homogeneous elements of degree $n$ of the free ${\mathcal N}$ algebra on one generator $x$. From the Definition \ref{def:$N$-algebra.}, it is clear that
the vector space $\mathcal{N}_n$ has a basis formed by all the elements of type:
$$N_r(M_1(\ldots),\ldots,M_{r-1}(\ldots),x),$$
where each $M_i$ is an element in the basis of ${\mathcal N}_{m_i}$ and $m_1+\cdots+m_{r-1}=n-1$. So, we get that $$|\mathcal{N}_n|=\sum |\mathcal{N}_{m_1}|\cdot \ldots \cdot |\mathcal{N}_{m_{r-1}}|$$ where the sum is taken over all the families $\{m_i\}_{1\leq i \leq r-1}$ such that $m_1+\cdots+m_{r-1}=n-1$.

In particular, $|\mathcal{N}_1|=1$ and, for $n>1$, we have that:
\[\begin{array}{rll}
|\mathcal{N}_n|&=\sum c_{m_1}\cdot \ldots \cdot c_{m_{r-1}}& \\
&=c_{n-1},\\
\end{array}\]
by a recursive hypothesis, which implies that $|\mathcal{N}_1|=c_0$ and that the integers $|\mathcal{N}_n|$ are defined by same equation than the Catalan numbers. We may conclude that $|\mathcal{N}_n|=c_{n-1}$, for $n\geq 1$.
\end{proof}

\begin{definition}\label{definition As N-algebra}
Given a compatible associative algebra $(A ,\pun ,\cir)$ ,
the $n$-ary operations $N_{n}:A^{\otimes n}\longrightarrow A$ on  $A$ are defined as follows:
$$N_n(x_1,\ldots , x_n)=(x_1\pun \ldots \pun x_{n-1})\cir x_n -x_1\pun ((x_2\pun \ldots \pun x_{n-1})\cir x_n)$$
\end{definition}

\begin{remark}\label{Observacion$N$en$As^2$}
The operations $N_n$ satisfy the following relations:
\begin{enumerate}
\item $N_2(x,y)=x \cir y - x \pun y$.
\item $N_n(x_1,\ldots , x_n)=N_3(x_1,x_2\pun \ldots \pun x_{n-1},x_n)$, for any $n \geq 4$.
\end{enumerate}
\end{remark}

\begin{proposition}\label{Prop$N$en$As^2$}
Let $(A ,\pun ,\cir)$ be a compatible associative algebra. For any family of elements $x_1, x_2, x_3 , x_4 , x_5 \in A$
we have that:
\[\begin{array}{l}
(1)\enspace N_2(N_2(x_1, x_2),x_3)=N_2(x_1,N_2(x_2,x_3)), \\
(2)\enspace N_3(x_1,x_2,N_2(x_3,x_4))=N_2(N_3(x_1,x_2,x_3),x_4)+N_3(x_1,N_2(x_2,x_3),x_4),\\
(3)\enspace N_3(x_1,x_2,N_3(x_3,x_4,x_5))=N_3(N_3(x_1,x_2,x_3),x_4,x_5)\\
\hspace{3cm} +N_4(x_1,N_2(x_2,x_3),x_4,x_5)-N_4(x_1,x_2,N_2(x_3,x_4),x_5)
\end{array}\]
\end{proposition}
\begin{proof}
The first relation states that $N_2$ is an associative product, which is true because $N_2$ is linear combination of the products
$\pun$ and $\cir$.\medskip

Let us prove the second statement. The proof of the other ones is obtained in an analogous way.

Let us denote the element $N_2(x,y)$ as $x\ast y=x\circ y-x\cdot y$. We have that,
\[\begin{array}{rcl} N_3(x_1,x_2,N_2(x_3,x_4) &=&(x_1\cdot x_2)\circ (x_3\ast x_4)-x_1\cdot(x_2\circ(x_3\ast x_4))\\
&=&(x_1\cdot x_2)\circ x_3\circ x_4-(x_1\cdot x_2)\circ (x_3 \cdot x_4)\\
& &-x_1\cdot (x_2\circ x_3\circ x_4)+x_1\cdot (x_2 \circ (x_3 \cdot x_4)).
\end{array}\]

Applying the compatibility condition between $\cdot$ and $\circ$, we get:

\begin{enumerate}
\item $(x_1 \cdot x_3)\circ (x_3 \cdot x_4)=((x_1 \cdot x_2)\circ x_3)\cdot x_4 - x_1 \cdot x_2 \cdot (x_3 \circ x_4) +(x_1 \cdot x_2 \cdot x_3)\circ x_4,$
\item $x_1\cdot (x_2\circ (x_3 \cdot x_4))=x_1 \cdot (x_2 \circ x_3 )\cdot x_4 -x_1\cdot x_2 \cdot (x_3 \circ x_4)+x_1 \cdot ((x_2 \cdot x_3 )\circ x_4).$
\end{enumerate}

So, we obtain that
\[\begin{array}{rcl} N_3(x_1,x_2,N_2(x_3,x_4) &=&(x_1\cdot x_2)\circ x_3\circ x_4-((x_1 \cdot x_2)\circ x_3)\cdot x_4\\
& &-(x_1 \cdot x_2 \cdot x_3)\circ x_4-x_1\cdot (x_2\circ x_3\circ x_4)\\
& &+x_1 \cdot (x_2 \circ x_3 )\cdot x_4 +x_1 \cdot ((x_2 \cdot x_3 )\circ x_4).
\end{array}\]

Regrouping the terms, we get:
\[\begin{array}{rcl} N_3(x_1,x_2,N_2(x_3,x_4) &=&N_3(x_1,x_2,x_3)\circ x_4-N(x_1,x_2,x_2)\cdot x_4\\
& &+(x_1\cdot (x_2 \ast x_3))\circ x_4-x_1 \cdot ((x_2 \ast x_3)\circ x_4\\
&=&N_2(N_3(x_1,x_2,x_3),x_4)+N_3(x_1,N_2(x_2,x_3),x_4),
\end{array}\]
which proves the equality.
\end{proof}

\begin{lemma}\label{lema$N$en$As^2$}
Let $(A, \cdot, \circ)$ be an compatible associative algebra and let $\{N_n\}_{n\geq 2}$ be the family of products introduced in Definition \ref{definition As N-algebra}. For  elements $x, y , z \in A$, we have that:
\begin{enumerate}
\item $N_2(x\pun y,z)=N_3(x,y,z)+x \pun N_2(y,z)$.
\item $N_2(x,y \pun z)=N_3(x,y,z)+N_2(x,y) \pun z$.
\end{enumerate}
\end{lemma}
\begin{proof}

The formulas are obtained by a straightforward computation, using the definition of the operations $N_n$s.
\end{proof}

The following result is immediate to prove.
\begin{proposition}\label{InduccionLemmaNen$As^2$}
Let $A$ be a compatible associative algebra $A$. For any family of elements $x_1, \ldots, x_n \in A$, we have that:
\begin{enumerate}
\item $N_2(x_1\pun \ldots \pun x_{n-1},x_n)=N_n(x_1,\ldots,x_n)
+x_1\pun N_{n-1}(x_2,\ldots,x_n)+ \ldots + x_1\pun \ldots \pun x_{n-2}\pun N_2(x_{n-1},x_n)$,
\item $N_2(x_1,x_2\pun \ldots \pun x_n)=N_n(x_1,\ldots,x_n)
+N_{n-1}(x_1,\ldots,x_{n-1})\pun x_n+ \ldots + N_2(x_1,x_2)\pun x_3\pun \ldots \pun x_n $,
\end{enumerate}
\end{proposition}

\begin{theorem}
 Let $(A ,\pun ,\cir)$ be a compatible associative algebra with $n$-ary operations
 $N_{n}$ introduced in Definition \ref{definition As N-algebra}. The data $(A,\{ N_n\})$ is a $\mathcal{N}$-algebra.
\end{theorem}

\begin{proof}
We apply Remark \ref{Observacion$N$en$As^2$} together with Proposition \ref{Prop$N$en$As^2$} and Proposition \ref{InduccionLemmaNen$As^2$}.

Let us prove the relation $(1)$ of Definition \ref{def:$N$-algebra.}. The proofs of the remaining relations follow by similar arguments.

Let $A$ be a compatible associative algebra and consider $x_1, \ldots,x_n,x_{n+1} \in A$, with $n\geq 3$.
The equality was proved in Proposition \ref{Prop$N$en$As^2$} for $n=3$. Let $n>3$, by Remark \ref{Observacion$N$en$As^2$}, we have that :
$$N_n(x_1,x_2,\ldots,x_{n-1},N_2(x_n,x_{n+1}))=N_3(x_1,x_2\cdot\ldots \cdot  x_{n-1},N_2(x_n,x_{n+1})).$$
So, from Proposition \ref{Prop$N$en$As^2$}, we obtain that

$$\begin{array}{ll} N_3(x_1,x_2\cdot \ldots \cdot  x_{n-1},N_2(x_n,x_{n+1})) &=N_2(N_3(x_1,x_2\cdot \ldots \cdot x_{n-1},x_n),x_{n+1})\\
& +N_3(x_1,N_2(x_2\cdot \ldots \cdot x_{n-1},x_n),x_{n+1}).
\end{array}$$

Applying Proposition \ref{Prop$N$en$As^2$} and Remark \ref{Observacion$N$en$As^2$} to the second term, we get:
$$\begin{array}{l} N_3(x_1,N_2(x_2\cdot \ldots \cdot x_{n-1},x_n),x_{n+1})\\
\hspace{3cm}=\displaystyle\sum_{i=2}^{n-1}N_3(x_1,x_2\cdot\ldots \cdot  x_{i-1}\cdot N_{n-i+1}(x_i, \ldots, x_n),x_{n+1})\\
\hspace{3cm}=\displaystyle\sum_{i=2}^{n-1}N_{i+1}(x_1,x_2,\ldots, x_{i-1}, N_{n-i+1}(x_i, \ldots, x_n),x_{n+1}).
\end{array}$$

To end the proof it suffices to apply Remark \ref{Observacion$N$en$As^2$} to first term, which implies that
$$ N_n(x_1,\ldots,N_2(x_n,x_{n+1}))=\displaystyle{ \sum_{i=1}^{n-1}N_{i+1}(x_1,\ldots,N_{n-i+1}(x_{i},\ldots,x_n),x_{n+1})}.$$

\end{proof}

\bigskip

\begin{theorem}\label{ElementosPrimComInfBialg}
Let $H$ be a compatible infinitesimal bialgebra with coproduct $\Delta$. If the elements $x_1, \ldots , x_n$ in $H$ are primitive, then $N_n(x_1,\ldots , x_n)$ is primitive, too. Therefore, we have that ${\mbox {\it Prim}(H)}$ is a $\mathcal{N}$-subalgebra of $(H, \{N_n\})$.
\end{theorem}
\begin{proof}
The cases $n=2$ and $n=3$ are obvious.

Suppose that  $n\geq 4$ and  that the elements $x_1, \ldots , x_n$ are primitive elements in $H$. Recall that
\[\begin{array}{l}(\ast)\enspace  N_n(x_1,\ldots , x_n)=  N_3(x_1,x_2\pun \ldots \pun x_{n-1},x_n)\\
\hspace{3cm}= (x_1\pun x_2\pun \ldots \pun x_{n-1})\cir x_n -x_1\pun ((x_2 \pun \ldots \pun x_{n-1})\cir x_n).\\
\end{array}\]
Applying a recursive argument on  $n\geq 2$, it is immediate to verify that
$$\Delta(x_1 \pun \ldots \pun x_n)= \displaystyle{ \sum_{i=1}^{n-1}(x_1\pun \ldots \pun x_i)\otimes (x_{i+1}\pun \ldots \pun x_n}).$$

Applying the formula above to $(\ast)$ we obtain that :
\[\begin{array}{rcl} \Delta((x_1\pun \ldots \pun x_{n-1})\cir  x_n) &=& \displaystyle{ \sum_{i=1}^{n-2}(x_1\pun \ldots \pun x_i)\otimes ((x_{i+1}\pun \ldots \pun x_{n-1}})\cir x_n) \\
&+& (x_1 \pun \ldots \pun x_{n-1})\otimes x_n,
\end{array}\]
and
\[\begin{array}{rcl} \Delta((x_2\pun \ldots \pun x_{r-1}) \cir x_r) &=&\displaystyle{ \sum_{i=1}^{n-2}(x_2\pun \ldots \pun x_i)\otimes ((x_{i+1}\pun \ldots \pun x_{n-1}})\cir x_n) \\
&+& (x_2 \pun \ldots \pun x_{n-1})\otimes x_n,
\end{array}\]
for $n\geq 3$.

Therefore, we may conclude that
\[\begin{array}{l} \Delta(x_1\pun ((x_2 \pun\ldots \pun x_{r-1}) \cir x_r)) =\displaystyle{ \sum_{i=2}^{n-2}(x_1\pun \ldots \pun x_i)\otimes ((x_{i+1}\pun \ldots \pun x_{n-1}})\cir x_n) \\
\hspace{3cm}+ (x_1 \pun \ldots \pun x_{n-1})\otimes x_n + x_1\otimes ((x_2\pun \ldots \pun x_{n-1})\cir x_n),
\end{array}\]
and thus $\Delta(N_n(x_1,\ldots, x_n))=0$, which ends the proof.
\end{proof}

\begin{remark}\label{descomposicionEn$As^2(V)$}
Let $V$ be a vector space. Denote by $H_V$ the free compatible associative algebra ${\rm As}^2(V)$ with the compatible infinitesimal bialgebra structure given in Proposition \ref{$As^2(V)$ ComoCompInfBialgebra}. By remark \ref{$As^2(v)$EslibreComoAlgAss}, $(H_V,\Delta, \cdot)$ is isomorphic as bialgebra to $\overline{T}({\rm Prim}(H_V)$. By identifying the product $\cdot$ with the concatenation product in $\overline{T}({\rm Prim}(H_V)$, we have tha any element $x \in H_V$ is written in a unique way as a linear combination of elements of the type $x_1\cdot \ldots \cdot x_r$, where $x_i \in {\rm Prim}(H_V)$, for each $1 \leq i  \leq r$.
\end{remark}

\begin{lemma}\label{PrimisgeneratedBy$V$}
Let $V$ be a vector space and $H_V$ the compatible infinitesimal algebra as in Remark \ref{descomposicionEn$As^2(V)$}. As $\mathcal{N}$-algebra, ${\rm Prim}(H_{V})$ is generated by $V$
\end{lemma}
\begin{proof}
Let $N_V$ be the sub-$\mathcal{N}$-algebra of $H_V$ generated by $V$. Since any element $x \in V$ is a primitive element, $N_V \subseteq {\rm Prim}(H_{V})$.

Let us see that any element $x \in H_V$ can written as a linear combination of elements of the type $x_1\cdot \ldots \cdot x_r$, where $x_i \in N_V$, for each $1\leq i \leq r$. It is sufficient to verify  the assertion for any tree $t \in H_V$. We prove the statement by induction on degree of $t$.\medskip

If $|t|=1$, then $t=a$, for some element $a$ in the basis $X$ of $H_V$. Suppose $|t|=n$, with $n>1$.

If $t$ is a reducible tree, then $t=t_1\cdot \ldots \cdot t_r$, where $r>1$ and $t_i$ is a tree of degree smaller than $n$, for each $1\leq i \leq r$. By a recursive argument, the assertion is true for each $t_i$, which implies the result for $t$.

Now, if $t$ is an irreducible tree, then $t=t'\circ a$, where $t'$ is a tree of degree $(n-1)$ and $a$ is an element in the basis  $X$. Applying a recursive argument, $t'$ is a linear combination of elements of the type $x_1\cdot \ldots \cdot x_r$, where $x_1, \ldots, x_r \in N_V$, with $1\leq r \leq n-1$. So, $t=t'\circ a$ is a linear combination of elements of the type $(x_1\cdot \ldots \cdot x_r)\circ a$, where $x_1, \ldots, x_r \in N_V$, with $1\leq r \leq n-1$ and $a$ an element of the basis $X$.

Now, consider the operation
$$N_{r+1}(x_1, \ldots ,x_r,a)=(x_1\cdot \ldots \cdot x_r)\circ a-x_1\cdot ((x_2 \cdot \ldots \cdot x_r)\circ a),$$
applied on the elements $x_1, \ldots ,x_r,a$. We have that
$$(x_1\cdot \ldots \cdot x_r)\circ a=N_{r+1}(x_1, \ldots ,x_r,a)+x_1\cdot ((x_2 \cdot \ldots \cdot x_r)\circ a).$$

Applying a recursive argument to $x_1$ and $(x_2 \cdot \ldots \cdot x_r)\circ a$, in the right side of the previous equality, we get the result.

By Remark \ref{descomposicionEn$As^2(V)$}, this implies that ${\rm Prim}(H_V)$ is generated, as $\mathcal{N}$-algebra by $N_V$, which implies that it is generated by $V$. This ends the proof.

\end{proof}

\begin{proposition}\label{$N$algebraLibre}
Let $V$ be vector space and let $H_{V}$ be the free associative compatible algebra ${\rm As}^2(V)$, spanned by $V$. The $\mathcal{N}$-algebra ${\rm Prim}(H_{V})$, of primitive elements of $H_{V}$, is the free $\mathcal{N}$-algebra on $V$.
\end{proposition}

\begin{proof}
Note that, by Lemma \ref{PrimisgeneratedBy$V$}, as $\mathcal{N}$-algebra, ${\rm Prim}(H_{V})$ is graded and generated by $V$. Denote by ${\rm Prim}(H_{V})_{n}$ the subspace of homogeneous elements of degree $n$ of ${\rm Prim }(H_V)$.\medskip

By Proposition \ref{Dimension $N$}, to see that ${\rm Prim}(H_{V})$ is the free $N$-algebra on $V$ it suffices to show that the dimension of ${\rm Prim}(H_{V})_{n}$ is equal to $(\mbox{dim}V)^{n}c_{n-1}$.\medskip

Let us compute the dimension of ${\rm Prim}(H_V)_{n}$. Recall from \cite{LodayRonco-On the structure of cofree Hopf algebras} the linear operator $e$. Since $(H_V, \cdot, \Delta)$ is a conilpotent infinitesimal bialgebra, we can define $e: H \rightarrow H$ given by
$$e(x)=x-x_{(1)}\cdot x_{(2)}-x_{(2)}\cdot x_{(2)}\cdot x_{(3)}+\cdots ,$$
where $\Delta(x)=x_{(1)}\otimes \cdots \otimes x_{(n)}$, for all $x \in H_V$. Consider the set $$B_n=\{e(t) | t \text{ is an irreducible tree of degree } n \text{ in } H_{V}\}.$$

Let us prove that the set $B_n$ is a basis of ${\rm Prim }(H_V)_{n}$.

From \cite{LodayRonco-On the structure of cofree Hopf algebras}, Proposition 2.5, we have that $B_n \subseteq e(H_{V})={\rm Prim}(H_{V})$ and for any reducible tree $t=t_1 \cdot \ldots \cdot t_r$,  $$e(t)=e(t_1\cdot \ldots \cdot t_r)=0.$$

So, $e({\mbox {Irr}}) =e(H_{V})={\rm Prim}(H_{V})$, because all element $x \in H_V$ can be written as a linear combination of elements in  $\bigcup _{n\geq 1}{\mbox {Irr}_n}$.

On the other hand, the same result asserts that, if $t$ is a irreducible tree of degree $n$, then
$$e(t)=t-t_{(1)}\cdot e(t_{(2)}).$$

So, if $t_1$ and $t_2$ are different irreducible trees in $H_{V}$, then $e(t_1)\neq e(t_2)$. In particular, since the number of irreducible trees of degree $n$ is equal to $(\mbox{dim}V)^{n}c_{n-1}$, we have that $|B_n|=(\mbox{dim}V)^{n}c_{n-1}$.\medskip

Let us see that the set $B_n$ is linearly independent. Note that in particular $|B_n|=c_{n-1}$. To simplify the notation, denote $l=|B_n|$ and let  $\{t^1, \ldots ,t^l\}$ be the set irreducible trees of degree $n$ in $H_{V}$. We have that  $B_n=\{e(t^1), \ldots ,e(t^l)\}$.\medskip

Let $\{\alpha_1,\ldots ,\alpha_{l} \}$ be  a family of elements in the field $\KK$ and suppose that $$\alpha_1e(t^1) +\ldots+\alpha_le(t^l)=0.$$
Since  $e(t^{i})=t^{i}- t^{i}_{(1)}\cdot e(t^{i}_{(2)})$, for any $1\leq i \leq l$, we have that
$$\alpha_1t^{1} +\ldots+\alpha_lt^{l}=\sum_{i=1}^{l}\alpha_i ( t^{i}_{(1)}\cdot e(t^{i}_{(2)})).$$

But this is  possible only if $\alpha_i=0$, for all $1 \leq i \leq l $, because the right side is linear combination of reducible trees. So, $B_n$ is linearly independent and we may conclude that it is a basis of ${\rm Prim}(H_{V})_{n}$. The  dimension of ${\rm Prim}(H_{V})_{n}$ is equal to $(\mbox{dim}V)^{n}c_{n-1}$, which ends the proof.

\end{proof}

Proposition \ref{$N$algebraLibre} and Proposition \ref{$As^2(v)$EslibreComoAlgAss} imply the following structure theorem.

\begin{theorem}\label{TeoremaDstructuraComInfBial}
Let $V$ be vector space and let $\mathcal{N}(V)$ be the free $\mathcal{N}$-algebra generated by $V$. The free associative compatible algebra ${\rm As}^2(V)$ is isomorphic to $T^{c}(\mathcal{N}(V))$.
\end{theorem}

\begin{remark}
Let $(A,\cdot,\circ,\Delta)$ be a compatible infinitesimal bialgebra. Consider $x\ast y=\alpha(x\cdot y)+\beta(x \circ y)$ a linear combination of the products $\cdot$ and $\circ$, where $\alpha$ and $\beta$ are elements in the field $\KK$.

A direct compute shows that
$$\Delta(x\ast y)=x_{(1)}\otimes x_{(2)} \ast y + x \ast y_{(1)}\otimes y_{(2)} + (\alpha+\beta )x \otimes y.$$

In particular, when $x\ast y=x\cdot y-x\circ y$, $(A, \cdot, \ast, \Delta)$ is a compatible associative algebra with coalgebra structure satisfying:
\begin{enumerate}
  \item $\Delta(x \cdot y)=x_{(1)}\otimes x_{(2)}\cdot y+x\cdot y_{(1)}\otimes y_{(2)}+x\otimes y,$
  \item $\Delta(x \ast y)=x_{(1)}\otimes x_{(2)}\ast y+x\ast y_{(1)}\otimes y_{(2)}$.
\end{enumerate}

So, $\Delta$ is infinitesimal unitary with respect to the product $\cdot$, in the Loday-Ronco's sense, and infinitesimal with respect to the product $\ast$, in the Joni-Rota's sense.

This notion of bialgebras is equivalent to the notion of bialgebras that we have given in Definition \ref{InfComp}.
\end{remark}

\bigskip

\section{Matching Dialgebras}\label{MatchingDialgebras}
In this section, we consider a particular case of compatible associative algebras, the {\it matching dialgebras}. In \cite{Zhang-The category and operad of matching dialgebras}, Y. Zhang, Ch. Bai and L. Guo studied the operad of matching dialgebras. They constructed the free matching dialgebras on a vector space $V$ by defining a matching dialgebra structure on the double tensor space $\overline{T}(\overline{T}(V))$.

In the same work, the authors proved that the operad of matching dialgebras is Koszul and compute the complex which gives the homology groups.\medskip

The aim of the present section is to study the notion of bialgebras in matching dialgebras.

Motivated by the {\it path Hopf algebra} $P(S)$ described by
A.B. Goncharov in \cite{Goncharov-Galois symmetries of fundamental groupoids and noncommutative geometry},
we introduced {\it bi-matching dialgebras}. We show that the Goncharov's Hopf algebras is part of a family of bi-matching dialgebras, which can be constructed from a bialgebra $(H,\cdot, \Delta)$(in the usual sense) and a {\it right semi-homomorphism} $R:H\rightarrow H$, which is a coderivation with respect to the coproduct $\Delta$.

We also develop the notion of compatible infinitesimal bialgebra in a matching dialgebras. In particular, a free matching dialgebra is a compatible infinitesimal bialgebra, which we obtain another example of a Loday\rq s good triple of operads (see \cite{Loday-GeneralizedBialgebras}).
\medskip
\begin{definition}
A {\it matching dialgebra} is a vector space $A$ with two associative products $\cdot$ and $\circ $ such that
$$(x \cdot y)\circ z=x\cdot(y \circ z) \text{, } (x\circ y) \cdot z = x\circ (y \cdot z)$$
for all $x , y, z \in A$.
\end{definition}
\medskip
We recall from \cite{Zhang-The category and operad of matching dialgebras} the notion of {\it right semi-homomorphism} of algebra.

This type of linear map gives an interesting family of examples of matching dialgebras.\medskip
\begin{definition}\label{semi-homomorfismo}
Let $(A,\cdot)$ be an associative algebra. A $\KK$-linear map $R:A\rightarrow A$ is a {\it right semi-homomorphism} if it satisfies the condition
$$R(x \cdot y)=R(x)\cdot y, \text{ for all } x, y \in A.$$
\end{definition}
\medskip
\begin{remark}\label{observacionSemi-homo}
Note that if $(A, \cdot)$ is an associative algebra and
$R:A\rightarrow A$ is a right semi-homomorphism, then $(A, \cdot, \circ)$ is a matching dialgebra
with the product $\circ: A \otimes A \rightarrow A $ given by $x\circ y:=x \cdot R(y)$ (see \cite{Zhang-The category and operad of matching dialgebras}).
\end{remark}
\medskip
\begin{example}\label{unidadEnUnSemi-Homm}
If $(A, \cdot)$ is an associative algebra and $a$ is an element in $A$, then the map $R:A \rightarrow A$ defined as $R(x):=a \cdot x$, for $x\in A$, is a right semi-homomorphism.

In particular, when $(A, \cdot)$ is a unital associative algebra with unit $e \in A$ and $R:A\rightarrow A$ is a right semi-homomorphism, then the linear map $R$ is completely determined by the action of $R$ on the unit $e$. Indeed, if $x \in A$, then $R(x)=R(e \cdot x)=R(e) \cdot x.$ So, for the case of unital associative algebra $A$, any right semi-homomorphism $R: A \rightarrow A$ is given by $R(x):=a \cdot x$, where $a$ is some element in $A$.
\end{example}

\bigskip

\subsection{The free matching dialgebra}
The free matching dialgebra over a vector space $V$ is a quotient of the free compatible associative algebra ${\rm As}^2(V)$. In particular, we may define an explicit compatible infinitesimal bialgebra structure on the free objects of the category of the matching dialgebras.

Given a vector space $V$, with basis $X$, let $T_{n}^{X}$ be  the set of planar rooted trees with $(n+1)$ vertices, whose non-root vertices are colored by the elements of $X$.

In the Subsection \ref{freeCompAssAlg}, we define a compatible associative algebra structure on the vector space spanned $\displaystyle \bigcup_{n\geq 1}T_{n}^{X}$ of colored planar rooted trees, where $X$ is a basis of $V$, and proved that ${\rm As}^2(V)$ is the free compatible associative algebra over $V$.

The free matching dialgebra over $V$  may be obtained as the quotient ${\rm As}^2(V)$, by the ideal spanned by the elements $(x \cdot y)\circ z=x\cdot(y \circ z)$ and $(x\circ y) \cdot z = x\circ (y \cdot z)$ , for $x,y$ and $z$ in $V$.

We want to find a set of trees which gives a set of representatives of the classes of ${\rm As}^2(V)$ modulo these relations.
\medskip

\begin{examples}\label{arbolesGradoBajo$As_2$}
In low degree,  we identify the trees:

\begin{figure}[h!]
  \centering
  \includegraphics[width=4.5cm]{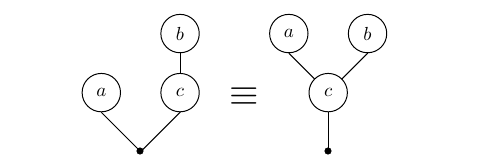}
\end{figure}

because $a\cdot(b \circ c)=(a \cdot b) \circ c$. In degree  four, we have that

\begin{figure}[h!]
  \centering
  \includegraphics[width=7cm]{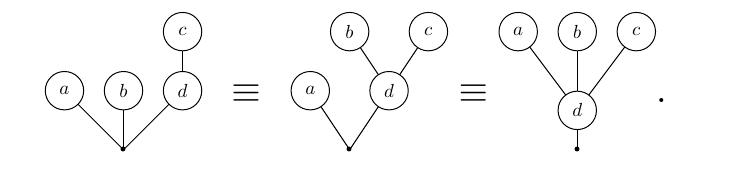}
\end{figure}

\end{examples}

In the general case, we have the following result.
\begin{proposition}\label{BaseDialgebra}
Any tree $t \in T_{n}^{X}${, for $n\geq 1$,} is equivalent to a tree of the type
$$t=t^1\cdot \ldots \cdot t^r,$$
where each $t^k$ is a tree of the form $t^k=a^k_1 \circ \ldots \circ a^k_{n_k}$, with $1\leq k \leq r$ and $n_1+\ldots +n_r=n$.
\end{proposition}

\begin{proof}
For  $n=3$, the result  was proved in \ref{arbolesGradoBajo$As_2$}. For $n>3$, suppose that the assertion is true for any tree of degree strictly less than $n$. If $t$ is an irreducible tree, then  $t=t'\circ a$, with  $|t'|=n-1$ and $ a $  an element of degree one. Applying a recursive argument to $t'$, we get
$$t=(t'_1\cdot \ldots \cdot t'_r)\circ a =t'_1\cdot \ldots \cdot (t'_r\circ a).$$

If $t$ is a reducible tree, then $t=t'\cdot t''$, where $t'$ and $t''$ are trees of degree strictly less than $n$. So, applying the inductive hypothesis to $t'$ and $t''$, we obtain the assertion for $t$, which ends the proof.
\end{proof}

\begin{notation}
We denote by $D_{n}^X$ the set of all trees  of degree $n$, described in Proposition \ref{BaseDialgebra}  and  by $D^X$ the set $\bigcup_{n\geq 1}D{}^{X}_{n}$.
\end{notation}
For instance, in degree three, we have that:\vspace{-.4cm}
\begin{figure}[h!]
  \centering
  \includegraphics[width=9.3cm]{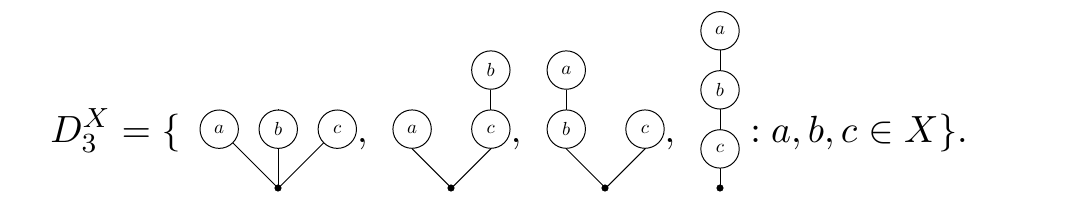}
\end{figure}\\

The free matching dialgebras is the vector space ${\rm{As}}_2(V)=\bigoplus_{n\geq 1}^{} \KK[D{}^{X}_{n}]$, whose basis is the set $D^X$.

Given two elements  $t=t^1 \cdot \ldots \cdot t^r$ and $w=w^1\cdot \ldots \cdot w^s$ in $D^X$ of degree $n$ and $m$, respectively, we have that $t \cdot w$ is the tree in $D_{n+m}^X$ that is obtained by identifying the  roots of $t$ and $w$, while that $t\circ w$ is the tree
$$t \circ w = t^1 \cdot \ldots \cdot t^{r-1}(t^r \circ w^1)\cdot w^2\ldots \cdot w^s,$$
where $t^r \circ w^1$ is the tree that is obtained by identify the root of $t^r$ with the only leaf of the tree $w^1$.
For instance,\begin{figure}[h!]
  \centering
  \includegraphics[width=8cm]{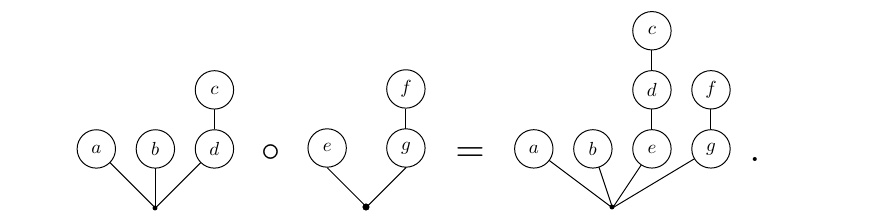}
\end{figure}

\begin{remark}
In \cite{Zhang-The category and operad of matching dialgebras}, Y. Zhang, C. Bai and  L. Guo defined the free matching dialgebra over the vector space $V$ as the double tensor space $\overline{T}(\overline{T}(V))$.

For a vector space $W$, $\overline{T}_{\ast}(W)=\displaystyle \bigoplus_{n\geq 1}W^{\otimes n}$ denotes the non-unitary tensor algebra, where the
tensor product is denoted by $\otimes_{\ast}$.

Under this notation, the double tensor space is $\overline{T}_{\ast_1}(\overline{T}_{\ast_2}(V))$. The products $\cdot$ and $\cdot$ are defined as follows:\\
For $u=u_1\otimes_{\ast_1}\cdots \otimes_{\ast_1}u_m$ and $v=v_1\otimes_{\ast_1}\cdots \otimes_{\ast_1}v_n$ in $\overline{T}_{\ast_1}(\overline{T}_{\ast_2}(V))$ with $u_i, v_j \in \overline{T}_{\ast_2}(V)$, for $1\leq i \leq m$, $1\leq j \leq n$, define:
\begin{enumerate}
\item $u \cdot v= u_1\otimes_{\ast_1}\cdots \otimes_{\ast_1}u_m \otimes_{\ast_1} v_1\otimes_{\ast_1}\cdots \otimes_{\ast_1}v_n$, the tensor product $\otimes_{\ast_1}$.
\item $u \circ v= u_1\otimes_{\ast_1}\cdots \otimes_{\ast_1}(u_m \otimes_{\ast_2} v_1)\otimes_{\ast_1}\cdots \otimes_{\ast_1}v_n.$
\end{enumerate}
In \cite{Zhang-The category and operad of matching dialgebras}, it is showed that $(\overline{T}_{\ast_1}(\overline{T}_{\ast_2}(V)), \cdot, \circ)$ is a matching dialgebra, which is free on the vector space $V$. The identification between both versions of the free matching dialgebras is clear. In our description, the tensors of first type are the trees of the type:
\begin{figure}[h!]
  \centering
  \includegraphics[width=3cm]{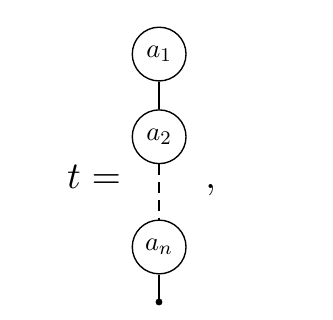}
\end{figure}\\
where $a_i \in X$, for $1 \leq i \leq n$, are element in the basis $X$. In this context, we will call {\it irreducible  tree } to the trees of this type.
\end{remark}

\newpage
\begin{remark}
For $V=\KK$, we identify the free matching dialgebra ${\rm As}_2(\KK)$ with the partition algebra $$C=\displaystyle\bigoplus_{n\geq 1}C^{n}, $$
where $C^n$ is vector space generated by all the ordered partition of a positive integer $n$. We denote by $c_{(n_1, \ldots,n_l)}$ the ordered partition $n=n_1+ \cdots + n_l$ of $n$. The products $\cdot$ and $\circ$ are given respectively by
$$c_{(n_1, \ldots,n_l)}\cdot c_{(m_1, \ldots,m_k)}:=c_{(n_1, \ldots,n_l,m_1, \ldots, m_k)}$$
and
$$c_{(n_1, \ldots,n_l)}\circ c_{(m_1, \ldots,m_k)}:=c_{(n_1, \ldots,n_l+m_1, \ldots, m_k)}.$$
Since the non-symmetric operad ${\rm As}_2$ is completely described by the free matching algebra on $V=\KK$, we have that $({\rm As}_2)_n=C^n$. In particular, the dimension of $({\rm As}_2)_n$ is $2^n$.
\end{remark}

\subsection{Bi-matching dialgebras}
We introduce the notion bi-matching dialgebras. For this, we requires the following proposition, which shows that the operad of matching dialgebras is a Hopf operad.
\begin{proposition}\label{morfismo de dialgebra}
If $(A, \cdot , \circ )$ is a matching dialgebras, then $A\otimes A$ is a matching dialgebras with the products defined by:
\begin{enumerate}
\item $(a_1\otimes a_2)\cdot (b_1 \otimes b_2)=a_1\cdot b_1 \otimes a_2\cdot b_2$ and
\item $(a_1\otimes a_2)\ast (b_1 \otimes  b_2)=a_1\cdot b_1 \otimes a_2\circ  b_2+a_1\circ b_1 \otimes a_2\cdot  b_2.$
\end{enumerate}
\end{proposition}
\begin{proof}
The proof follows by direct computation.
\end{proof}
\medskip

\begin{remark}\label{observacionsobre morfismos de dialgebras}
Note that in Proposition \ref{morfismo de dialgebra}, the associativity of the product $\ast$ requires the compatibility condition between the products $\cdot $ and $\circ$.\\
\end{remark}

The following notion of bialgebra was originally introduced by A.B. Goncharov in \cite{Goncharov-Galois symmetries of fundamental groupoids and noncommutative geometry}.

\begin{definition}
A bi-matching dialgebra is a matching dialgebra $(H,\cdot,\circ)$ equipped with a coassociative coproduct $\Delta: H \rightarrow H\otimes H$ such that $\Delta$ is morphism of matching dialgebras with respect to the matching dialgebra structure of $H\otimes H$ defined in Proposition \ref{morfismo de dialgebra}.
\end{definition}

\begin{proposition}
Let $(H, \cdot, \Delta)$ be a bialgebra and let $R:H \rightarrow H$ be a right semi-homomorphism. If $R$ is a coderivation for the product $\cdot$, then
$$\Delta(x \circ y)=\Delta(x)\ast \Delta(y),$$
for any $x,y\in H$, where $x\circ y=x \cdot R(y)$ is the product defined in Remark \ref{observacionSemi-homo}.
\end{proposition}
\begin{proof} By a straightforward calculation, we get:
\[\begin{array}{rll} \Delta(x\circ y)&=\Delta(x \cdot R(y))& \\
&=\Delta(x)\cdot \Delta(R(y))\\
&=x_{(1)}\otimes x_{(2)}\cdot (R(y_{(1)})\otimes y_{(2)}+ y_{(1)}\otimes R(y_{(2)}))\\
&=x_{(1)}\cdot R(y_{(1)}) \otimes x_{(2)}\cdot y_{(2)} + x_{(1)} \cdot y_{(1)}\otimes x_{(2)}\cdot  R(y_{(2)})\\
&=x_{(1)}\circ y_{(1)} \otimes x_{(2)}\cdot y_{(2)} + x_{(1)} \cdot y_{(1)}\otimes x_{(2)}\circ y_{(2)}\\
&=\Delta(x)\ast \Delta(y),
\end{array}\]
which proves the formula.
\end{proof}

\medskip

\begin{example}\label{SemiHomoEnG-L}
Consider the Grossman-Larson's Hopf algebra $H=\KK[\mathcal{T}]$ with basis the set of all non-planar rooted trees $\mathcal{T}$ described in \cite{Grossman-larson_Hopf-algebraic structure of families of trees}. Recall that the tree $e$ with one vertex is the unit for the product defined in $H$. Consider the linear map $R:H\rightarrow H$ such that, for any rooted tree $t$, $R(t)$ is the sum of trees  obtained  from $t$ by attaching one more outgoing edge and vertex to each vertex of $t$, which is originally defined on the Connes-Kreimer's Hopf algebra in \cite{Connes-Kreimer HopfAlgebras}.

In \cite{Panaite-Relating G-L y C-K}, Proposition 2.2, F. Panaite showed the linear map $R$ is a right semi-homomorphism for $H$. In fact, $R(x)=R(e)\cdot x$, for all $x \in H$. In his work  F. Panaite showed that $R$ is a coderivation for the coproduct $\Delta$ defined in $H$. Indeed, since $R(e)$ is a primitive element, we have that
\[\begin{array}{rll} \Delta(R(x))&=\Delta(R(e) \cdot x)& \\
&=\Delta(R(e))\cdot \Delta(x)\\
&=(R(e)\otimes e +e \otimes R(e))\cdot x_{(1)}\otimes x_{(2)}\\
&=R(e)\cdot x_{(1)}\otimes x_{(2)}+x_{(1)}\otimes R(e)\cdot x_{(2)}\\
&=R( x_{(1)})\otimes x_{(2)}+x_{(1)}\otimes R( x_{(2)}),\\
\end{array}\]
and $R$ is a coderivation. So, $(H, \cdot, \circ, \Delta)$ is a bialgebra, where $\circ$ is the associative product induced by $R$ and the compatibility condition between the products $\cdot$ and $\circ$ with the coproduct $\Delta$ is as in Remark \ref{observacionsobre morfismos de dialgebras}.
\end{example}

\begin{remark}\label{R(e)Primitivo}
The previous result obtained by F. Panaite may be generalized to any bialgebra $(H, \cdot, \Delta)$ with unit $e \in H$, that is, if $R:H\rightarrow H$ is a right semi-homomorphism and $R(e)$ is a primitive element of $H$, then $R$ is a coderivation. The proof is similar to that given in the Example \ref{SemiHomoEnG-L}.
\end{remark}

\begin{example}
Let $H=\KK[X]$ be the $\KK$-algebra of polynomial in one variable, with the usual product and the coproduct given by:
$$\Delta (X^n):=\sum _{i=0}^n\binom{n}{i}X^{n-i}\otimes X^{i},$$ with the homomorphism $R$ defined by  $R(X^n)=X^{n+1}$. As $R(1)=X$ is a primitive element, we get that $R$ is a  coderivation.
\end{example}

\subsection{The Goncharov's Hopf algebra} Let us describe the path algebra $P(S)$, introduced by A. B. Goncharov in  \cite{Goncharov-Galois symmetries of fundamental groupoids and noncommutative geometry}, which  motivates our notion of bialgebra, described in Remark \ref{observacionsobre morfismos de dialgebras}. \medskip

Let $S$ be a finite set. Denote by $P(S)$ the $\KK$-vector space with basis
$$p_{s_0, \ldots , s_n}, \quad {\rm for}\ n\geq 1,\ {\rm and}\  s_k \in S, \text{ for } k=0, \ldots ,n.$$

The associative product $\cdot: P(S) \otimes P(S) \rightarrow P(S)$ is defined as follows:
$$p_{a,X,b} \cdot p_{c, Y, d}=\begin{cases}
 p_{a,X,Y,d} & \text{, for}\ b=c,\\
  0 & \text{, for}\  b\neq c,\\
 \end{cases}$$
where the letters $a,b,c,d$ denote elements, and $X$ and $Y$ denote sequences, possibly empty, of elements of the set $S$. In particular, $p_{a,b}=p_{a,x}\cdot p_{x,b}$, for $x\in S$, and the unit for this product is  the element $e=\sum_{i \in S}p_{i,i}$.

The coproduct $\Delta: P(S) \rightarrow P(S) \otimes P(S)$ is given by:
$$\Delta(p_{a,x_1, \ldots ,x_n,b})=\sum_{k=0}^{n}\sum_{\sigma \in Sh(k,n-k)}p_{a,x_{\sigma(1)}, \ldots ,x_{\sigma (k)}, b}\otimes p_{a,x_{\sigma(k+1)}, \ldots ,x_{\sigma (n)}, b}.$$

For instance, $\Delta(p_{a,b})=p_{a,b}\otimes p_{a,b}$, for $a,b\in S$, and $$\Delta(e)=\sum_{i\in S}p_{i,i}\otimes p_{i,i}\neq e\otimes e.$$
The linear map $R:P(S)\rightarrow P(S)$, given by: $$R(e)=\sum_{i \in S}p_{i,i,i}$$
is a right semi-homomorphism.

With the definition above, we get that $R(p_{a,X,b})=p_{a,a,X,b}$, for any element $p_{a,X,b}$ of the basis. So, $R$ induces a new associative product $\circ :P(S)\otimes P(S) \rightarrow P(S)$ by setting $x\circ y=x\cdot R(y)$, that is:
$$p_{X,b} \circ p_{c,Y}=\begin{cases}
 p_{X,b,Y}, & {\rm for}\ b=c,\\
  0, & {\rm for}\ b\neq c,\\
 \end{cases}$$
 where $b, c \in S$, and $X$ and $Y$ are sequences of elements of $S$.
 \medskip
\begin{proposition}
The right semi-homomorphism $R:P(S) \rightarrow P(S)$ is a coderivation.
\end{proposition}
\begin{proof}
Note that for any element $i \in S$, we have that $\Delta(p_{i,i,i})=p_{i,i,i}\otimes p_{i,i} + p_{i,i}\otimes p_{i,i,i},$ therefore:
 $$\Delta(R(e))=\sum_{i\in S}p_{i,i,i}\otimes p_{i,i} + p_{i,i}\otimes p_{i,i,i}.$$

Let $x=p_{a,X,b}$  be an element of the basis of $P(S)$. By definition of the coproduct $\Delta$, we have that the element $\Delta(x)$ is a sum of tensors of type $$p_{a,X',b}\otimes p_{a,X'',b},$$ where $X'$ and $X''$ are (possibly empty) ordered subsequences of $X$.

Using the Sweedler' notation, we write
$$\Delta(x)=x_{(1)}\otimes x_{(2)}= p_{a,X_{(1)},b}\otimes p_{a,X_{(2)},b}.$$

Computing $\Delta(R(x))$, we obtain that:
\[\begin{array}{rll} \Delta(R(x))&=\Delta(R(e) \cdot x)& \\
&=(\sum_{i\in S}p_{i,i,i}\otimes p_{i,i} + p_{i,i}\otimes p_{i,i,i})\cdot x_{(1)}\otimes x_{(2)}\\
&=(\sum_{i\in S}p_{i,i,i}\otimes p_{i,i} + p_{i,i}\otimes p_{i,i,i})\cdot p_{a,X_{(1)},b}\otimes p_{a,X_{(2)},b}\\
&=p_{a,a,a}\cdot p_{a,X_{(1)},b}\otimes p_{a,a}\cdot p_{a,X_{(2)},b} + p_{a,a}\cdot p_{a,X_{(1)},b}\otimes p_{a,a,a}\cdot p_{a,X_{(2)},b} \\
&= p_{a,a,X_{(1)},b}\otimes  p_{a,X_{(2)},b} +  p_{a,X_{(1)},b}\otimes p_{a,a,X_{(2)},b} \\
&=R(x_{(1)})\otimes x_{(2)} + x_{(1)}\otimes R(x_{(2)}), \\
\end{array}\]
which ends the proof.
\end{proof}

\medskip

\subsection{Notion of compatible infinitesimal bialgebra in matching dialgebras}

We consider the notion of compatible infinitesimal bialgebra in matching dialgebras. A direct compute shows that this notion of bialgebra is well-defined in a matching dialgebra.\medskip

Let $(A, \circ, \Delta)$ be an infinitesimal bialgebra. The product $\circ$ and the coproduct $\Delta$ may be extended to $\overline{T}(A)=\bigoplus _{n\geq 1}A^{\otimes n}$ as follows:
\begin{enumerate}
\item $(a_1\ldots a_n) \circ (b_1\ldots b_m) = a_1\ldots a_{n-1}(a_n \circ b_1)b_2\ldots b_m$ and
\item $\Delta(a_1\ldots a_n)=\displaystyle \sum_{i=1}^{n-1}a_1\ldots a_{i-1}\Delta(a_i)a_{i+1}\ldots a_n+ \displaystyle\sum_{i=1}^{n-1}a_1\ldots a_{i}\otimes a_{i+1}\ldots a_n$
\end{enumerate}
If we denoted by $\cdot$ the concatenation product in $\overline{T}(A)$, then $(\overline{T}(A), \cdot, \circ)$ is a matching dialgebra, and $\Delta$ is infinitesimal for both products.

In particular, consider the free matching dialgebra ${\rm As}_2(V)=\overline{T}(\overline{T}(V))$. In this case, $A=\overline{T}(V)$ is an infinitesimal bialgebra with the concatenation product and the deconcatenation coproduct. Identifying the tree $a_1\circ \ldots \circ a_n$  with a tensor in $\overline{T}(V)$ and the product $\circ$ with the concatenation product, we get:
$$\Delta (a_1\circ \ldots \circ a_n):=\sum _{i=1}^{n-1}(a_1\circ \ldots \circ a_i)\otimes (a_{i+1}\circ \ldots \circ a_n).$$

Thus, extending $\Delta$ to $\overline{T}(\overline{T}(V))$, we have that ${\rm As}_2(V)$ is a compatible infinitesimal bialgebra.\medskip

The explicit formula for the coproduct $\Delta$ is given by:
$$\Delta(t)=\sum_{i=1}^{n-1}t_{\{a_1,\ldots,a_i\}}\otimes t_{\{a_{i+1},\ldots,a_{n}\}}, $$
described in Proposition \ref{FormulaCoproducto}, which extends the deconcatenation coproduct of $\overline{T}(V)$.
\newpage
\begin{example}
When $t$ is the tree
\begin{figure}[h!]
  \centering
  \includegraphics[width=5.5cm]{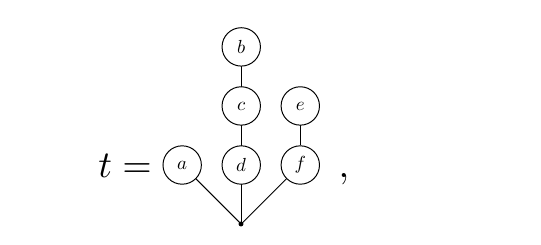}
\end{figure}\\
the coproduct $\Delta(t)$ is given by:
\begin{figure}[h!]
  \centering
  \includegraphics[width=12cm]{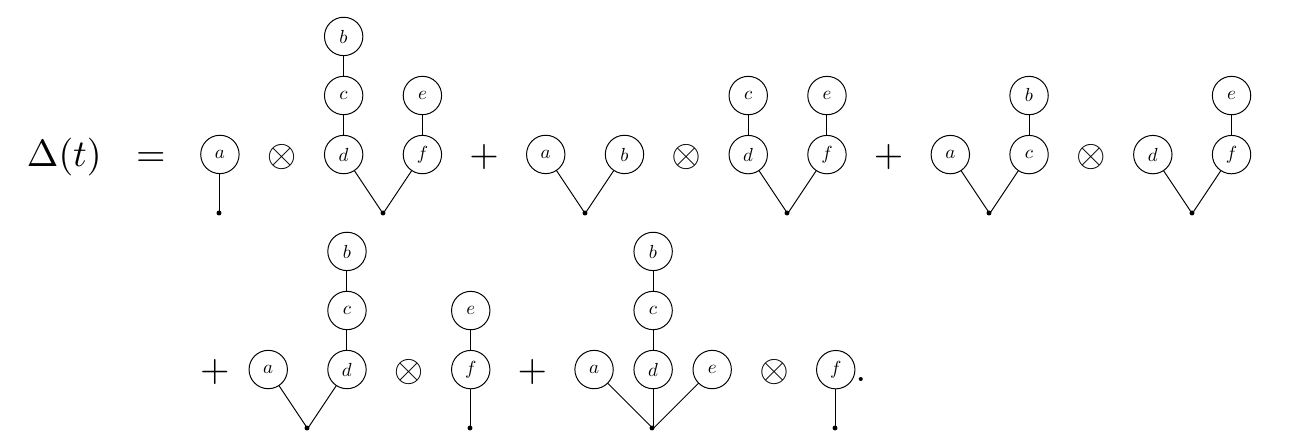}
\end{figure}

\end{example}
\begin{remark}
The primitive part of the compatible infinitesimal bialgebra ${\rm As}_2(V)$ is generated by the associative product $\ast$ given by $x \ast y=x\circ y - x\cdot y$ together with the $\mathcal{N}$-operations of superior degree. By the compatibility conditions defining a matching dialgebras, the $\mathcal{N}$-algebra structure of ${\
rm As}_2(V)$ is reduced only associative product $\ast$. Thus, we get a good triple of operads $( {\rm As}, {\rm As}_2, {\rm As})$, in the Loday's sense (see \cite{Loday-GeneralizedBialgebras}).
\end{remark}

\end{document}